\documentclass[11pt,reqno]{amsart}
\usepackage{enumerate, latexsym, amsmath, amsfonts, amssymb, amsthm, color}
\usepackage{mathrsfs}
\usepackage{booktabs,array}
\usepackage{enumitem}
\usepackage[hidelinks]{hyperref}
\def\pmod #1{\ ({\rm{mod}}\ #1)}
\def\Z{\mathbb Z}
\def\N{\mathbb N}

\def\Q{\mathbb Q}

\def\l{\left}
\def\r{\right}
\def\bg{\bigg}
\def\({\bg(}
\def\){\bg)}
\def\t{\text}
\def\f{\frac}

\def\cs{\ldots}

\def\ls{\leqslant}
\def\gs{\geqslant}
\def\se {\subseteq}
\def\sm{\setminus}

\def\al{\alpha}

\def\ga{\gamma}
\def\ve{\varepsilon}

\def\da{\delta}
\def\la{\lambda}

\def\va{\varphi}

\newcommand{\Acal}{\mathcal A}
\newcommand{\Cscr}{\mathscr C}
\newcommand{\Sq}{\mathop{\rm Sq}}
\newcommand{\Cat}{\mathop{\rm Cat}}

\newcommand{\abs}[1]{\lvert#1\rvert}

\theoremstyle{plain}
\newtheorem{theorem}{Theorem}[section]

\newtheorem{lemma}{Lemma}[section]
\newtheorem{corollary}{Corollary}[section]
\newtheorem{proposition}{Proposition}[section]

\theoremstyle{definition}

\theoremstyle{remark}
\newtheorem{remark}{Remark}

\hypersetup{pdftitle={Hilbert's tenth problem over the integers with few unknowns},pdfauthor={Zhi-Wei Sun}}
\begin{document}

\hbox{Preprint, {\tt arXiv:2609.01594}}
\medskip

\title
[{Hilbert's tenth problem with few unknowns}]
{On Hilbert's tenth problem over the integers with few unknowns}

\author
[Zhi-Wei Sun] {Zhi-Wei Sun}

\address{School of Mathematics, Nanjing
University, Nanjing 210093, People's Republic of China}
\email{zwsun@nju.edu.cn}

\keywords{Hilbert's tenth problem, Diophantine equations, integral solutions, undecidability, universal pairs.
\newline \indent Supported by the Natural Science Foundation of China (grant no. 12371004).}

\begin{abstract}
Hilbert's Tenth Problem over the integers was solved negatively by Matiyasevich in 1970.
In this paper we prove that there is no algorithm to decide for any polynomial equation with integer coefficients and 9 unknowns whether it has integer solutions, which improves the previous 11 unknowns theorem. We also obtain such as a result for polynomial equations with 8 natural number unknowns, which improves Matiyasevich's 9 unknowns theorem obtained in 1975. A key ingredient in our new approach is the use of a generating function to encode both a power of two and the divisibility of a central binomial coefficient by means of a single Lucas term. We also use Catalan numbers to prove that $(4,47)$ is a universal pair in the sense of Jones, improving his quartic universal pair $(4,58)$.  We further obtain explicit universal degree bounds $50048\cdot5^{47}+11456$ for 8 natural number unknowns, and $200192\cdot5^{47}+44296$ for 9 integer unknowns.
\end{abstract}
\maketitle

\section{Introduction}
\setcounter{equation}{0}

Hilbert's Tenth Problem (HTP) asks for an effective algorithm to test whether
an arbitrary polynomial equation
$$P(z_1,\cs,z_n)=0$$
(with integer coefficients) has solutions over the ring $\Z$ of integers.
This is the tenth problem of his famous 23 mathematical problems posed at the 1900 ICM.

Let $\N=\{0,1,2,\ldots\}$. A subset ${\mathcal A}$ of $\N$ is called a {\it Diophantine set} if
there is a polynomial $P(x_0,x_1,\ldots,x_n)$ with integer coefficients
such that for any $a\in\N$ we have
$$a\in {\mathcal A}\iff\exists x_1\gs0\ldots\exists x_n\gs0\,[P(a,x_1,\ldots,x_n)=0].$$
(Throughout this paper, variables always range over $\Z$.) 

Based on the important work of M. Davis, H. Putnam and J. Robinson \cite{DPR} on exponential Diophantine equations, in 1970 Y. Matiyasevich \cite{M70} solved HTP negatively by proving Davis' hypothesis which asserts that Diophantine sets coincide with r.e. (recursively enumerable) sets arising from the theory of computability. Note that some r.e. sets are not recursive (cf. \cite[pp.\,140-141]{C80}).

For convenience, for a set $S\se\Z$ and a fixed positive integer $n$, we let $\exists^n$ over $S$ denote the set
$$\{\exists x_1\in S\ldots\exists x_n\in S[P(x_1,\ldots,x_n)=0]:\ P(x_1,\ldots,x_n)\in \Z[x_1,\ldots,x_n]\}.$$
It is easy to see that $\exists$ over $\N$ and $\exists$ over $\Z$ are decidable. In fact, 
 if $a_0,a_1,\ldots,a_n$ and $z$ are integers
with $a_0z\not=0$ and $\sum_{i=0}^na_iz^{n-i}=0$, then
$$|z|^n\ls|a_0z^n|\ls\sum_{i=1}^n|a_i|\cdot|z|^{n-i}\ls \sum_{i=1}^n|a_i|\cdot|z|^{n-1}\ \t{and hence}\ |z|\ls \sum_{i=1}^n|a_i|.$$

Matiyasevich's 9 unknowns theorem announced in \cite{M77} asserts that $\exists^9$ over $\N$
is undecidable. For a complete proof of this, the reader may consult Jones \cite{J82}.

By the Gauss-Legendre theorem on sums of three squares (cf. \cite[pp.\, 17-23]{N96}), the number $4m+1$ with $m\in\N$ can be written as the sum of
two even squares and an odd square. It follows that for any integer $m$ we have
\begin{equation}\label{1.1}m\gs0\iff \exists x\exists y\exists z[m=x^2+y^2+z^2+z].\end{equation}
Therefore the undecidability of $\exists^n$ over $\N$ implies
the undecidability of $\exists^{3n}$ over $\Z$.
In 1992 Z.-W. Sun \cite{S92b} showed for any $n\in\Z^+=\{1,2,3,\ldots\}$ that if $\exists^n$ over $\N$ is undecidable then so is $\exists^{2n+2}$ over $\Z$.

In 1985 S. P. Tung \cite{T85} asked for small values of $n$ such that $\exists^n$ over $\Z$
is undecidable.  In 1992, the author announced in \cite{Sun92, S92b} that $\exists^{11}$ over $\Z$ is undecidable, the whole sophisticated proof of which was published in \cite{Sun21}.
For applications of the 11 unknowns theorem, see \cite{RGK,MPR,MS}.

For the extended HTP over a ring $R$ containing $\Z$,
the usual strategy to obtain its undecidability is as follows: Prove that $\Z$ is Diophantine over $R$ and then use the result that HTP over $\Z$ is undecidable.
Thus, to find a small positive integer $m$ with $\exists^m$ over $R$ undecidable, depends heavily on
the undecidability of $\exists^n$ over $\Z$ ({\it not} $\N$) with $n$ as small as possible. In this sense, to find a small number
$n$ with $\exists^{n}$ over $\Z$ undecidable is quite important and very useful.

Now we state our first main result.

\begin{theorem}\label{Th1.1}
Let $\Acal\se\N$ be any r.e. set.
\begin{enumerate}[label=\textup{(\roman*)}]
\item There is a polynomial $P_{\Acal}(z_0,z_1,\ldots,z_8)$ with integer coefficients such that for any $a\in\N$ we have
\begin{equation}\label{eq:mixed-main}
 a\in\Acal\iff\exists z_1\ldots\exists z_7\exists z_8\gs0\,
 [P_{\Acal}(a,z_1,\ldots,z_8)=0].
\end{equation}
\item There is a polynomial $Q_{\Acal}(z_0,z_1,\ldots,z_9)$ with integer coefficients such that for any $a\in\N$ we have
\begin{equation}\label{eq:integer-main}
 a\in\Acal\iff\exists z_1\ldots\exists z_9\,
 [Q_{\Acal}(a,z_1,\ldots,z_9)=0].
\end{equation}
\end{enumerate}
\end{theorem}

Since there are nonrecursive r.e. sets,
and $P(z)=0$ for some integer $z$ if and only if $P(z)P(-z)=0$ for some $z\in\N$, Theorem \ref{Th1.1} has the following consequence.

\begin{corollary}\label{Cor1.1}
There is no algorithm to decide whether an arbitrary polynomial equation with integer coefficients in eight unknowns has a solution in $\Z^7\times\N$. Furthermore, $\exists^8$ over $\N$ and $\exists^9$ over $\Z$ are undecidable.
\end{corollary}

By a simple trick of  Putnam \cite{P69} (see also \cite[Lemma 6.2(i)]{Sun21}), for any polynomial $P(x)\in\Z[x]$, we have
\begin{equation*}\N\cap\{(x+1)(1-P(x)^2)-1:\ x\in\N\}=\{x\in\N:\ P(x)=0\}.
\end{equation*}
As the set $\mathcal P$ of all primes is recursive,
using his 9 unknowns theorem Matiyasevich \cite{M81} constructed a polynomial
$P(x_1,\ldots,x_{10})$ with integer coefficients such that
$$\mathcal P=\N\cap\{P(x_1,\ldots,x_{10}):\ x_1,\ldots,x_{10}\in\N\}.$$
Similarly, by taking $\mathcal A=\mathcal P$, we obtain from Theorem \ref{Th1.1}(ii) 
and Putnam's trick the following corollary.

\begin{corollary}\label{Cor1.2} For the set $\mathcal P$ of all primes, there is a polynomial
$P(x_1,\ldots,x_{9})$ with integer coefficients such that
$$\mathcal P=\N\cap\{P(x_1,\ldots,x_{9}):\ x_1,\ldots,x_{9}\in\N\}.$$
\end{corollary}

Skolem \cite{Sk} reduced any integer-coefficient polynomial equation over $\N$ or $\Z$
to one of degree at most four.
It is also interesting to study simultaneous controls of the number of unknowns and the degree
in Diophantine representations. 

Fix an effective numbering $(\varphi_a)_{a\in\N}$ of the unary partial computable functions, as in \cite[Chapter~4, Definition~2.1, p.\,77]{C80}, and put
\[
 W_a=\operatorname{Dom}(\varphi_a)=\{x\in\N:\varphi_a(x)\text{ is defined}\}\quad(a\in\N).
\]
Thus $W_0,W_1,\ldots$ enumerate all r.e. subsets of $\N$ (with repetitions), and the relation $x\in W_a$ is r.e. in $(a,x)$ (cf. \cite[Chapter~5, Section~1 and Chapter~7, Theorem~2.3, p.\,124]{C80}).

We adopt the following effective form of the index-parameter convention of Jones \cite{J82}. We call $(\delta,m)$ a \emph{universal pair} if there exist a fixed positive integer $r$, an integer polynomial
\[
 U(x,c_1,\ldots,c_r,z_1,\ldots,z_m),
\]
and total computable functions $c_1(a),\ldots,c_r(a)$ from $\N$ to $\Z$, such that, for all $a,x\in\N$,
\begin{equation}\label{eq:universal-definition}
 x\in W_a\iff\exists z_1>0\ldots\exists z_m>0\,
 [U(x,c_1(a),\ldots,c_r(a),z_1,\ldots,z_m)=0].
\end{equation}
We require that the total degree of $U$ in $x,z_1,\ldots,z_m$, regarding $c_1,\ldots,c_r$ as coefficients, be at most $\delta$. Equivalently, there is a computable family of integer polynomials $U_a(x,z_1,\ldots,z_m)$ of total degree at most $\delta$ representing $W_a$ in this way: one may take the coefficients in the finite list of monomials of degree at most $\delta$ as the index parameters. In particular, the degree in the input $x$ is counted, whereas the number and degrees of the index parameters, computed from $a$ before solving the equation, are not counted. This degree convention is that of \cite[Sections~1 and~5]{J82}, where a triple of index parameters is used. We explicitly require effective dependence on $a$, but do not require the functions $c_i(a)$ themselves to be polynomials in $a$.

Jones \cite[Theorem~4 and Section~5]{J82} proved the quartic universal pair $(4,58)$. We improve the number of unknowns as follows.

\begin{theorem}\label{Th1.2}
The pair $(4,47)$ is universal.
\end{theorem}

The proof uses a low-degree version of the single Lucas term method. In place of a quotient of Lucas terms we extract a Catalan coefficient from a square root of a quadratic expression. We also pack the digit tests of a universal quadratic source system into a smaller number of blocks, using the behavior of signed digits under borrowing. Theorem \ref{Th1.2} will be proved in Section \ref{sec:quartic}.

For $S=\N$ or $\Z$, we call $(\delta,m)$ a universal pair over $S$ if the same definition holds with $z_1,\ldots,z_m\in S$ in \eqref{eq:universal-definition}. The positive- and nonnegative-integer witness conventions are equivalent by translation, without increasing the degree. Applying our quartic universal pair to the small-unknown constructions, with additional degree optimizations, gives the following explicit bounds.

\begin{corollary}\label{Cor1.2}
The pair $(50048\cdot5^{47}+11456,8)$ is universal over $\N$, and the pair $(200192\cdot5^{47}+44296,9)$ is universal over $\Z$.
\end{corollary}
These are upper bounds on the degrees of the constructed universal polynomials; no minimality assertion is made. The proof is in Section \ref{sec:degree-bounds}.

Throughout this paper, we adopt the notations
\[
 \Sq(\Z)=\{r^2:r\in\Z\},\ \ \t{and}\ \ 
 p\uparrow:=\{p^n:n\in\N\}\ \ \t{for }p\in\Z^+.
\]
For any prime $p$, the $p$-adic order of a nonzero integer $m$ is denoted by $\nu_p(m)$. We use $\Z^*=\Z\sm\{0\}$.

The paper is organized as follows. Section \ref{sec:prelim} records the Lucas identities, the relation-combining polynomials and the source-coding input. Sections \ref{sec:modified} and \ref{sec:interpolation} give a modified two-witness Lucas certificate and prove residue-class interpolation for its second witness. Section \ref{sec:merger} records a polynomial used to ensure nonvanishing. In Section \ref{sec:bridge} we establish a new binomial-coefficient bridge involving a single specified Lucas term. Section \ref{sec:final} proves both parts of Theorem \ref{Th1.1}. Section \ref{sec:quartic} gives the quadratic system and proves Theorem \ref{Th1.2}, and Section \ref{sec:degree-bounds} proves Corollary \ref{Cor1.2}.

\section{Some preparations}\label{sec:prelim}
\setcounter{equation}{0}

Let $A$ be an integer. We define the Lucas sequences $u_n=u_n(A)$ and $v_n=v_n(A)$ $(n\in\Z)$ by letting
\begin{equation}\label{1.18}
 u_0=0,\ u_1=1,\quad u_{n-1}+u_{n+1}=Au_n\ \ \t{for all}\ n\in\Z,
\end{equation}
and
\begin{equation}\label{1.19}
 v_0=2,\ v_1=A,\quad v_{n-1}+v_{n+1}=Av_n\ \ \t{for all}\ n\in\Z.
\end{equation}
It is easy to see that
\begin{equation}\label{1.20}
 u_{-n}(A)=-u_n(A)=(-1)^nu_n(-A),\quad
 v_{-n}(A)=v_n(A)=(-1)^nv_n(-A).
\end{equation}
Put $\la=(A+\sqrt{A^2-4})/2$. By the Binet formulas, we have
\begin{equation}\label{eq:binet}
 u_n(A)=\begin{cases}\f{\la^n-\la^{-n}}{\la-\la^{-1}}&\t{if}\ A\not=\pm2,
 \\n(A/2)^{n-1}&\t{if}\ A\in\{\pm2\},\end{cases}\ \ \t{and}\ \ 
 v_n(A)=\la^n+\la^{-n}.
\end{equation}
In particular,
\begin{equation}\label{uv}
 u_{2n}(A)=u_n(A)v_n(A)\ \ \t{and}\ \ v_n(A)^2-(A^2-4)u_n(A)^2=4
\end{equation}
(cf. \cite[Theorems 1.2.3 and 1.2.7]{S24}). For $A>2$, the sequence $(u_j(A))_{j\gs0}$ is strictly increasing and
\begin{equation}\label{eq:growth}
 (A-1)^j\ls u_{j+1}(A)\ls A^j\quad(j\in\N).
\end{equation}
These inequalities follow by induction from \eqref{1.18}; see also \cite[Lemma~4.1]{Sun21}. We shall use repeatedly the congruence
\begin{equation}\label{eq:index-cong}
 u_j(A)\equiv u_j(2)=j\pmod{A-2}\quad(A\ne2,\ j\in\Z).
\end{equation}
For the matrix
\begin{equation}\label{MA}
 M_A=\begin{pmatrix}A&-1\\1&0\end{pmatrix},
\end{equation}
we have $\det M_A=1$ and
\begin{equation}\label{eq:matrix}
 M_A^n=\begin{pmatrix}u_{n+1}(A)&-u_n(A)\\u_n(A)&-u_{n-1}(A)\end{pmatrix}
 =u_n(A)M_A-u_{n-1}(A)I_2
\end{equation}
for $n\in\Z$. Here $I_2$ denotes the identity matrix of size two.

For the relation $C=u_B(A)$ with $A,B,C\in\Z$, the author studied its Diophantine representations over $\Z$ in \cite{Sun92}, which laid the foundation of the proof of the 11 unknowns theorem in \cite{Sun21}.

\begin{lemma}\label{lem:pell-lucas}
For $A,X\in\Z$, we have
\[
 (A^2-4)X^2+4\in\Sq(\Z)\iff X=u_n(A)\ \t{for some }n\in\Z.
\]
\end{lemma}
\begin{remark}
This is \cite[Lemma~4.3]{Sun21}, which extends
\cite[Lemma~9]{Sun92} to all integral parameters.
\end{remark}

\begin{lemma}\label{lem:index-separation} Let $A,n,s,t,k\in\Z$ satisfy $\abs A>2$, $n>3$, and
\[
 v_n(\abs A)=2k>0.
\]
 If  $u_s(A)\equiv u_t(A)\pmod k$, then
\[
 s\equiv t\pmod{2n}\quad\t{or}\quad s\equiv-t\pmod{2n}.
\]
\end{lemma}
\begin{remark}
This is the integral form of \cite[Lemma~10]{Sun92}.
\end{remark}

The following result appeared in \cite[Lemma 7]{Sun92} and the basic idea came from Matiyasevich \cite{M70}.

\begin{lemma}\label{lem:square-divisibility}
Let $A$ be an integer with $|A|\gs2$.  For any $r,n\in\N$, we have 
\[
 u_r(A)^2\mid u_n(A)
\iff ru_r(A)\mid n.\]
\end{lemma}

\begin{lemma}\label{lem:valuation}
Let $A>2$ be even, and let $n\in\Z^+$. For any prime $p$ dividing $u_n(A)$, we have
\begin{equation*}
 \nu_p(u_n(A))\gs\nu_p(n).
\end{equation*}
\end{lemma}
\begin{proof}
Let $\rho$ be the order of the class of $M_A$ given by \eqref{MA} in the projective general linear group $\operatorname{PGL}_2(\Z/p\Z)$. By \eqref{eq:matrix}, $\rho$ is the least positive index $j$ for which $p\mid u_j(A)$, and $\rho\mid n$. If the two eigenvalues of $M_A$ modulo $p$ are distinct, then the order of their quotient divides either $p-1$ or $p+1$. If they coincide, $M_A$ is a nonscalar Jordan block up to similarity and multiplication by a scalar, so $\rho=p$. In particular $\nu_p(\rho)\ls1$ in all cases.

If $p\mid u_j(A)$, expand $(cI_2+u_j(A)M_A)^p$ with $c=-u_{j-1}(A)$. All terms other than the constant scalar term have entries divisible by $p u_j(A)$, and the lower left entry is $u_{pj}(A)$. Thus $\nu_p(u_{pj}(A))\gs\nu_p(u_j(A))+1$. It is known that $u_j(A)\mid u_{lj}(A)$ for all $l\in\Z^+$
(cf. \cite[Theorem 1.2.9]{S24}). Writing $n=\rho l$ with $l\in\Z^+$, we obtain
\[
 \nu_p(u_n(A))\gs 1+\nu_p(l)\gs\nu_p(n).
\]
This concludes the proof. 
\end{proof}

The following relation-combining polynomial is due to Matiyasevich and Robinson \cite{MR75}. For $s\in\Z^+$ and $\mathbf X=(X_1,\ldots,X_s)$, put
\begin{equation*}
 U=1+\sum_{j=1}^sX_j^2\ \,\t{and}\ \,
 J_s(\mathbf X,z)=\prod_{\ve_1,\ldots,\ve_s\in\{\pm1\}}
 \bigg(z+\sum_{j=1}^s\ve_j\sqrt{X_j}U^{j-1}\bigg).
\end{equation*}
As the product is invariant under each change of sign of the formal square roots, it expands to an integer polynomial. In particular, it is monic in $z$ of degree $2^s$. Define the integer polynomials
\begin{equation}\label{eq:Hk}
 H_s(\mathbf X,S,T,m)=S^{2^s}J_s\left(\mathbf X,m-\f{T}{S}\right)
\end{equation}
and
\begin{equation}\label{eq:Mk}
 \mathcal M_s(\mathbf X,R,m)=
 \prod_{\ve_1,\ldots,\ve_s\in\{\pm1\}}
 \bigg(m-(2R-1)\bigg(U^s+\sum_{j=1}^s\ve_j\sqrt{X_j}U^{j-1}\bigg)\bigg).
\end{equation}
The right-hand side of \eqref{eq:Hk} is interpreted after expansion, also when $S=0$.

\begin{lemma}\label{lem:relation-combining}
Let $A_1,\ldots,A_s,R,S,T\in\Z$ with $S\ne0$. We have
\begin{align}
 & A_1,\ldots,A_s\in\Sq(\Z)\ \land\ S\mid T\notag\\
 &\qquad\iff\exists m\,[H_s(A_1,\ldots,A_s,S,T,m)=0],\label{eq:H-equivalence}\\
 & A_1,\ldots,A_s\in\Sq(\Z)\ \land\ R>0\notag\\
 &\qquad\iff\exists m>0\,[\mathcal M_s(A_1,\ldots,A_s,R,m)=0].\label{eq:M-equivalence}
\end{align}
Moreover,
\begin{equation}\label{eq:H-zero}
 H_s(\mathbf X,0,T,m)=T^{2^s}.
\end{equation}
\end{lemma}
\begin{proof}
Equivalence \eqref{eq:H-equivalence} is \cite[Lemma~5.3]{Sun21}, and \eqref{eq:M-equivalence} is the specialization $S=1$ and $T=0$ of \cite[Lemma~5.1]{Sun21}. These are the relation-combining results of \cite{MR75}. Since $J_s$ is monic of degree $2^s$ in its final argument, setting $S=0$ after homogenization in \eqref{eq:Hk} gives \eqref{eq:H-zero}.
\end{proof}

The following lemma is essentially \cite[Lemma~5.5]{Sun21}.

\begin{lemma}\label{lem:positivity}
For every $R\in\Z$, we have
\begin{equation*}
 R>0\iff\exists z\ne0\,[(3R-4)z^2+1\in\Sq(\Z)].
\end{equation*}
\end{lemma}

We now give the source-coding input in the particular form needed in this paper. Fix an r.e. set $\Acal\se\N$, which is Diophantine by Matiyasevich's theorem \cite{M70}. For $a\in\N$ and $f\in\Z$, put
\begin{equation}\label{eq:b-def}
 d_a=3(2a+1), \ \widehat f=2f+d_af^2\ \t{and}\ 
 b=(1+d_af)^2=1+d_a\widehat f.
\end{equation}
Clearly, $b\gs1$. Also, $\widehat f>0$ and $b>1$ if $f\ne0$.

\begin{proposition}\label{prop:source-coding}
There are effectively constructible integer polynomials $X=X(a,f,g)$, $Y=Y(a,f,g)$ and $\Cscr=2^{\alpha_1}b^{\alpha_2}$, with fixed positive integers $\alpha_1,\alpha_2$, having the following properties for each $a\in\N$.
\begin{enumerate}[label=\textup{(\roman*)}]
\item If $a\in\Acal$, there exist $f>0$ and $g\in\Z$ such that
\[
 b\in2\uparrow,\qquad b\ls g<\Cscr,\qquad Y\mid\binom{2X}{X}.
\]
The witness $f$ can be chosen arbitrarily large.
\item If $f\ne0$, $0\ls g<2\Cscr$, $b\in2\uparrow$ and $Y\mid\binom{2X}{X}$, then $a\in\Acal$.
\item We have $Y>0$ for all $f,g\in\Z$. If $f\ne0$ and $0\ls g<2\Cscr$, then
\begin{equation}\label{eq:source-size}
 4\mid X,\quad X\gs3b,\quad Y\gs\max\{b,2^8\},\quad Y\ls X+1.
\end{equation}
If in addition $b\in2\uparrow$, then $Y\in2\uparrow$.
\end{enumerate}
\end{proposition}

\begin{proof}
Use the source polynomials in the proof of \cite[Theorem~3.1]{Sun21} with $p=2$, substituting $\widehat f$ for the original variable $f$ in that proof. The reparametrization \eqref{eq:b-def} makes $b$ a square. In the positive direction the cited proof chooses $b=2^{2n\va(d_a)}$ with an arbitrarily large positive integer $n$. Since $d_a$ is odd, Euler's theorem gives $2^{n\va(d_a)}\equiv1\pmod{d_a}$; thus this choice has the form \eqref{eq:b-def} with $f>0$. Assertions (i), (ii) and the two lower bounds for $X,Y$ in (iii) now follow from \cite[(3.1)--(3.4)]{Sun21}.

For clarity we verify the extra properties of the particular source polynomials. Use $B',N',S',T',R'$ for the quantities $B,N,S,T,R$, respectively, in the source construction on pp.\,290--292 of \cite{Sun21}. If $\da$ is its fixed degree bound and $\nu$ its fixed number of variables besides the input, then
\[
 B'=2^{2\al}b^\da\ \ \t{and}\ \ N'=4(B')^{2(\da+1)^{\nu+1}+2},
\]
where $\al\in\Z^+$ is fixed. In particular, $4\mid B'$ and $B'\mid N'$. The constant digit of the mask in \cite[(3.7)]{Sun21} is $B'-1$. The remaining summand in $T'$ is divisible by $B'$, so $B'\mid T'+1$. 
Set $R'=(S'+T'+1)N'+T'+1$. 
The definitions of the source polynomials are
\[
 X=(N'-1)R\ \ \t{and}\ \ Y=(N')^2.
\]
Hence $Y>0$, $4\mid X$, and $Y$ is a power of two whenever $b$ is. Under the hypotheses of (iii), the source construction yields $S',T'\gs0$ (see \cite[pp.\,291--292]{Sun21}), so $R'\gs N'+1$. Therefore $X\gs (N')^2-1=Y-1$, as desired.
\end{proof}

\section{A modified Lucas certificate}\label{sec:modified}
\setcounter{equation}{0}

Let $A,B,C,\xi,y\in\Z$.  Define
\begin{align}
 D&=(A^2-4)C^2+4,\label{eq:D}\\
 E&=C^2D\xi,\label{eq:E}\\
 F&=4(A^2-4)E^2+1,\label{eq:F}\\
 G^*&=2(A+2)(A-2)^2E^2-1,\label{eq:Gstar}\\
 \Omega&=16E,\notag\\
 H^*&=C+(B-C)F+\Omega Fy,\label{eq:Hstar}\\
 I^*&=((G^*)^2-1)(H^*)^2+1.\label{eq:Istar}
\end{align}

We use $G^*,H^*,I^*$ to replace $G,H,I$ in Sun's papers \cite{Sun92, Sun21}.
The following lemma can be easily checked.

\begin{lemma}\label{lem:structural-identities} We have
\begin{equation*}
 \gcd(F,\Omega)=1.
\end{equation*}
If $2\mid A$ and $E\ne0$, then
\begin{align}
 2G^*&=(A-2)F-A,\label{eq:key-identity}\\
 F&\equiv1\pmod{\Omega\xi},\label{eq:F-cong}\\
 2G^*&\equiv-2\pmod{\Omega\xi}.\label{eq:G-cong}
\end{align}
\end{lemma}

The next theorem replaces the reverse implication of
\cite[Lemma~4.6]{Sun21}.

\begin{theorem}\label{thm:modified-reverse}
Let $A,B,C\in\Z$ satisfy
\begin{equation}\label{eq:reverse-range}
 A\gs6,\quad 2\mid A,
 \qquad
 2\nmid B,
 \qquad
 1<B<\frac A2-1,
 \ \ \t{and}\ \ 
 A-2\mid C-B.
\end{equation}
If there exist $\xi,y\in\Z$ with $\xi\neq0$ such that
\begin{equation}\label{eq:DFI-square}
 DFI^*\in\Sq(\Z),
\end{equation}
then
 $C=u_B(A).$
\end{theorem}
\begin{proof}
First $C\neq0$:
otherwise $A-2\mid B$, whereas $0<B<A-2$. Hence $E\ne0$. By \eqref{eq:Gstar}, $G^*>1$ and $D,F,I^*>0$.

Since $D\mid E$, we have
\[
 F\equiv1\pmod D,\qquad G^*\equiv-1\pmod D,\qquad I^*\equiv1\pmod D.
\]
Also, by \eqref{eq:key-identity} and $H^*\equiv C\pmod F$, we have
\[
 4I^*=((2G^*)^2-4)(H^*)^2+4\equiv D\pmod F.
\]
Thus the positive integers $D,F,I^*$ are pairwise coprime, since $F$ is odd.
The hypothesis \eqref{eq:DFI-square} therefore makes each of them a square.
By Lemma \ref{lem:pell-lucas} there are $s,\mu,t\in\Z$ such that
\begin{equation*}
 C=u_s(A),\qquad 4E=u_\mu(A),\qquad H^*=u_t(2G^*).
\end{equation*}

Because $C\neq0$, also $s\neq0$.  Put
\[
 n=\abs\mu,\qquad r=2\abs s,\qquad U=\abs{u_r(A)}.
\]
By \eqref{uv},
\[
 C^2D=u_s(A)^2v_s(A)^2=u_{2s}(A)^2=U^2.
\]
As $4E=4C^2D\xi=u_\mu(A)$ and $\xi\neq0$, we have
$u_r(A)^2\mid u_n(A)$.  By Lemma 
\ref{lem:square-divisibility},
\begin{equation}\label{eq:strong-index-divisibility}
 rU\mid n.
\end{equation}
We have $U=u_r(A)\gs r=2\abs s$ and $U\gs A$ by \eqref{eq:growth}. In particular $n>0$.
Furthermore,
\[
 4F=(A^2-4)u_\mu(A)^2+4=v_n(A)^2.
\]
Since $A$ is even, $v_n(A)$ is even and positive. Thus, for $k=v_n(A)/2$, we have $F=k^2$.

The congruence $A-2\mid C-B$ shows that $C$ is odd. For any even integer $T\in\Z$, we have $u_j(T)\equiv u_j(2)=j\pmod2$ for all $j\in\Z$. Hence $s$ is odd;
\eqref{eq:Hstar} then shows that $H^*$, and therefore $t$, is odd.  In view of  \eqref{eq:key-identity}, we have
\[
 u_s(A)=C\equiv H^*=u_t(2G^*)\equiv u_t(-A)=u_t(A)\pmod F.
\]
In particular this holds modulo $k$, so Lemma \ref{lem:index-separation} yields
\[
 s\equiv t\pmod{2n}\quad\text{or}\quad s\equiv-t\pmod{2n}.
\]
Since \eqref{eq:strong-index-divisibility} implies $U\mid n$, we have
\begin{equation}\label{eq:s-t-mod-U}
 s\equiv t\pmod{2U}\quad\text{or}\quad s\equiv-t\pmod{2U}.
\end{equation}

On the other hand, by Lemma \ref{lem:structural-identities} and the oddness of $t$
we have
\[
 H^*\equiv B\pmod\Omega\ \ \t{and}\ \ 
 H^*=u_t(2G^*)\equiv u_t(-2)=t\pmod\Omega.
\]
As $\Omega=16E=16U^2\xi$, we have $2U\mid\Omega$ and 
$t\equiv B\pmod{2U}$.  This, together with \eqref{eq:s-t-mod-U}, implies that
\begin{equation}\label{eq:s-pm-B-U}
 s\equiv B\pmod{2U}\quad\text{or}\quad s\equiv-B\pmod{2U}.
\end{equation}
But $U=u_r(A)\gs r=2\abs s$, while $U\gs A$ and $B<A/2-1<U/2$.  Therefore
$\abs s+B<U$.  The relevant integer $s\mp B$ in
\eqref{eq:s-pm-B-U} is a multiple of $2U$ with absolute value smaller than $U$, and hence it is
zero.  Thus $s=B$ or $s=-B$.

Finally, note that
$C=u_s(A)\equiv u_s(2)=s\pmod{A-2}$, while $C\equiv B\pmod{A-2}$.
If $s=-B$, then $A-2\mid2B$, contrary to
$2B<A-2$.  So $s=B$ and $C=u_B(A)$.
\end{proof}

\section{Full residue-class interpolation}\label{sec:interpolation}
\setcounter{equation}{0}

The positive direction supplies more than existence of two Lucas witnesses.
It allows the second witness to be prescribed modulo the first.

\begin{theorem}[Full residue-class interpolation]\label{thm:full-interpolation}
Let $A>2$ be even, let $B>0$ be odd, and put
\[
 C=u_B(A).
\]
For each $N\in\Z^+$, there is a positive integer $\xi$ divisible
by $N$ with the following property: For any $c\in\Z$ there exists $y\in\Z$ such that with the definitions
\eqref{eq:D}--\eqref{eq:Istar}, we have
\[
 y\equiv c\pmod\xi
 \qquad\text{and}\qquad
 DFI^*\in\Sq(\Z).
\]
\end{theorem}

\begin{proof}
Since $C=u_B(A)$, by \eqref{uv} we have
$D=v_B(A)^2.$
Set
\[
 M=4NC^2D.
\]
The matrix $M_A$ given by \eqref{MA} is invertible modulo $M$.  Its class therefore has finite
order in $\operatorname{GL}_2(\Z/M\Z)$.  Choose $n\in\Z^+$ such that
$M_A^n\equiv I_2\pmod M$.  
Clearly, $M\mid u_n(A)$ by \eqref{eq:matrix}.  Define
\begin{equation*}
 \xi=\frac{u_n(A)}{4C^2D},
 \qquad
 E=\frac{u_n(A)}4,
 \qquad
 r=\frac{v_n(A)}2.
\end{equation*}
Then $N\mid\xi$, and the expressions in \eqref{eq:E} and
\eqref{eq:F} agree with these choices.  In particular,
\begin{equation*}
 F=r^2.
\end{equation*}
Since $A$ is even, $u_n(A)\equiv u_n(2)=n\pmod2$. As $4\mid u_n(A)$, the index $n$ is even.  Moreover,
$v_n(A)=v_{n/2}(A)^2-2\equiv2\pmod4$, so $r$ is odd.

We next construct an explicit period modulo $F$.  Put $N_A=M_A^n$.  Then
$\operatorname{tr}(N_A)=2r$ and $\det N_A=1$. By the Cayley--Hamilton theorem in linear algebra,
\[
 N_A^2-2rN_A+I_2=0.
\]
Since $r$ is odd, the binomial theorem gives
\[
 N_A^{2r}=(-I_2+2rN_A)^r\equiv-I_2\pmod{r^2},
\]
and hence
\begin{equation}\label{eq:period}
 M_A^{\lambda}\equiv I_2\pmod F,
 \ \t{with}\ 
 \lambda=4nr.
\end{equation}

Recall that $\Omega=16E=4u_n(A)$, and set
\begin{equation*}
 d=\frac{\lambda}{\gcd(\lambda,\Omega)}.
\end{equation*}
We claim that
\begin{equation}\label{eq:d-xi-coprime}
 \gcd(d,\xi)=1.
\end{equation}
Indeed, \eqref{eq:F} and \eqref{eq:E} show that
$r^2=F\equiv1\pmod{\xi^2}$, so $\gcd(r,\xi)=1$.  If a prime $p$ divides $\xi$, then Lemma \ref{lem:valuation} gives $\nu_p(u_n(A))\gs \nu_p(n)$.  Therefore
\[
 \nu_p(\lambda)=\nu_p(4n)
 \leq \nu_p(4u_n(A))=\nu_p(\Omega).
\]
Every prime-power factor of $\lambda$ supported on $\xi$ is removed by
$\gcd(\lambda,\Omega)$, proving \eqref{eq:d-xi-coprime}.

Fix a residue class $c\pmod\xi$.  By the Chinese Remainder Theorem, choose
an arbitrarily large positive integer $L$ such that
\begin{equation*}
 d\mid L \ \ \t{and}\ \ 
 L\equiv c\pmod\xi.
\end{equation*}
The definition of $d$ implies
\begin{equation}\label{eq:lambda-divides}
 \lambda\mid\Omega L.
\end{equation}
Clearly,
$ j=B+\Omega L$ is positive and odd.  Define
\[
 H_0=u_j(2G^*).
\]
By \eqref{eq:key-identity}, 
\eqref{eq:period} and \eqref{eq:lambda-divides}, we have
\[
 H_0\equiv u_j(-A)=u_j(A)
 \equiv u_B(A)=C\pmod F.
\]
In view of \eqref{eq:G-cong},
\[
 H_0\equiv u_j(-2)=j\equiv B\pmod\Omega.
\]
As $\gcd(F,\Omega)=1$, the integer
\begin{equation*}
 y=\frac{H_0-C-(B-C)F}{\Omega F}
\end{equation*}
is well defined, and $H^*=H_0$.

It remains to compute $y$ modulo $\xi$.  Working modulo $\Omega\xi$ and
using \eqref{eq:F-cong} and \eqref{eq:G-cong}, we obtain
\[
 H_0\equiv u_j(-2)=j=B+\Omega L\pmod{\Omega\xi},
\]
whereas
\[
 C+(B-C)F\equiv B\pmod{\Omega\xi}.
\]
Consequently,
\[
 \Omega Fy\equiv\Omega L\pmod{\Omega\xi}.
\]
After dividing the integer divisibility relation by $\Omega$ and using
$F\equiv1\pmod\xi$, we get
\[
 y\equiv L\equiv c\pmod\xi.
\]
Finally, $D=v_B(A)^2$, $F=r^2$, and \eqref{uv} applied to
$H_0=u_j(2G^*)$ gives
\[
 I^*=\left(\frac{v_j(2G^*)}{2}\right)^2.
\]
Thus $DFI^*$ is a square.
\end{proof}

\begin{remark}
The essential point in Theorem \ref{thm:full-interpolation} is
\eqref{eq:d-xi-coprime}.  The factor $r$ in the explicit period is already
coprime to $\xi$, and $\Omega=4u_n(A)$ absorbs the entire $\xi$-supported
part of the remaining factor $4n$. 
\end{remark}

\section{A polynomial for nonvanishing}\label{sec:merger}
\setcounter{equation}{0}

Define
\begin{equation}\label{eq:Pi}
 \Pi(X)=(X^2-13)(X^2-17)(X^2-221).
\end{equation}

\begin{lemma}\label{lem:universal-root}
The polynomial $\Pi$ has no zero in $\Z$.  For every nonzero integer $q$,
there is $c\in\Z$ such that
\[
 q\mid\Pi(c).
\]
\end{lemma}

\begin{proof}
The numbers $13$, $17$ and $221$ are not integral squares, so $\Pi$ has no
integral zero.

Let $p$ be an odd prime.  If $p\notin\{13,17\}$, then either $13$ or $17$
is a quadratic residue modulo $p$, or both are nonresidues; in the latter
case $221=13\cdot17$ is a residue.  Thus one factor of $\Pi$ has a root
modulo $p$.  For $p=13$, use $17\equiv2^2\pmod{13}$; for $p=17$, use
$13\equiv8^2\pmod{17}$.  These roots are nonzero and hence simple, so they
lift to roots modulo every power $p^e$.

For powers of $2$, the congruence $x^2\equiv17\pmod{2^e}$ is solvable for
all $e\geq1$.  It is immediate for $e\leq3$.  If $x^2\equiv17\pmod{2^e}$
with $e\geq3$ and $x$ odd, one of $x$ and $x+2^{e-1}$ gives a solution
modulo $2^{e+1}$.  The Chinese Remainder Theorem now gives a root modulo
$\abs q$.
\end{proof}

\section{A new binomial-coefficient bridge}\label{sec:bridge}
\setcounter{equation}{0}

The proof of the 9 unknowns theorem in \cite{J82} and the proof in \cite{Sun21} use a quotient of two Lucas terms with different parameters to extract a binomial coefficient. We shall instead use a generating function, with only one specified Lucas term. The power of two required by Proposition \ref{prop:source-coding} will be recovered from the same extraction.

\begin{lemma}\label{lem:single-lucas}
Let $b,X\in\Z^+$, and for $j\in\N$ and $\ve\in\{\pm1\}$ put
\[
 d_j^{(\ve)}=2^{j+1}+\ve b\binom{2j}{j}.
\]
Define the integer polynomial in $A$ (for the given $b,X$)
\begin{equation}\label{eq:extraction-poly}
 P_X^{(\ve)}(A)=d_X^{(\ve)}+A\sum_{j=1}^X d_{X-j}^{(\ve)}u_j(A).
\end{equation}
Then
\begin{equation}\label{eq:constant-term}
 P_X^{(\ve)}(0)=2^{X+1}+\ve b\binom{2X}{X}.
\end{equation}
For $A\gs10$, let $\la=(A+\sqrt{A^2-4})/2$ and $C=u_{X+1}(A)$. Define
\begin{equation}\label{eq:alpha}
 \al_\ve=AC\left(\f2{\la-2}+\f{\ve b}{\sqrt{\la(\la-4)}}\right)
\end{equation}
Then
\begin{equation}\label{eq:approx}
 \left|\al_\ve-P_X^{(\ve)}(A)\right|<\f{12(b+2)4^X}{\la}.
\end{equation}
\end{lemma}
\begin{proof}
Equation \eqref{eq:constant-term} is immediate from \eqref{eq:extraction-poly}. The generating function of $d_j^{(\ve)}$ is
\begin{equation*}
 \mathcal G_\ve(t)=\sum_{j=0}^{\infty}d_j^{(\ve)}t^j
 =\f{2}{1-2t}+\f{\ve b}{\sqrt{1-4t}}\quad (|t|<1/4).
\end{equation*}
Put $\omega=A/(\la-\la^{-1})=(1+\la^{-2})/(1-\la^{-2})$. By the Binet formulas \eqref{eq:binet},
\[
 \al_\ve=\omega(\la^X-\la^{-X-2})\mathcal G_\ve(1/\la)\ \ \t{and}\ \ 
 Au_j(A)=\omega(\la^j-\la^{-j}).
\]
Subtracting \eqref{eq:extraction-poly} gives
\begin{align*}
 \al_\ve-P_X^{(\ve)}(A)
 ={}&(\omega-1)d_X^{(\ve)}
    +\omega\sum_{j>X}d_j^{(\ve)}\la^{X-j}\\
 &+\omega\sum_{j=1}^X d_{X-j}^{(\ve)}\la^{-j}
    -\omega\la^{-X-2}\mathcal G_\ve(1/\la).
\end{align*}
Since $|d_j^{(\ve)}|\ls(b+2)4^j$ and $\la>8$, we have $\omega<65/63$ and $0<\omega-1<1/\la$. Also,
\begin{align*}
 \left|\sum_{j>X}d_j^{(\ve)}\la^{X-j}\right|
 &\ls(b+2)4^X\f4{\la-4}<\f{8(b+2)4^X}{\la},\\
 \left|\sum_{j=1}^X d_{X-j}^{(\ve)}\la^{-j}\right|
 &<\f{(b+2)4^X}{4\la-1}<\f{(b+2)4^X}{3\la},\\
 \left|\la^{-X-2}\mathcal G_\ve(1/\la)\right|
 &\ls\f{(b+2)\la^{-X-2}}{1-4/\la}<\f{(b+2)4^X}{\la}.
\end{align*}
Combining these bounds yields \eqref{eq:approx}.
\end{proof}

\begin{remark}
The polynomial $P_X^{(\ve)}(A)$ in \eqref{eq:extraction-poly} is used only to identify a nearby integer. Its degree depends on $X$, which in the applications is itself a variable. Neither this polynomial nor the generating function occurs in our final Diophantine equations.
\end{remark}

\begin{lemma}\label{lem:merged-arithmetic}
Suppose that $b,X,Y$ are as in Proposition \ref{prop:source-coding}, with $f\ne0$ and $0\ls g<2\Cscr$. Let $bY\mid A$. For either choice of $\ve\in\{\pm1\}$,
\begin{equation}\label{eq:merged-arithmetic}
 bY\mid P_X^{(\ve)}(A)
 \iff b\in2\uparrow\ \land\ Y\mid\binom{2X}{X}.
\end{equation}
\end{lemma}
\begin{proof}
Equation \eqref{eq:constant-term} gives
\begin{equation}\label{eq:merged-congruence}
 P_X^{(\ve)}(A)\equiv2^{X+1}+\ve b\binom{2X}{X}\pmod{bY}.
\end{equation}
If $bY\mid P_X^{(\ve)}(A)$, reduction modulo $b$ shows that $b\mid2^{X+1}$, so $b\in2\uparrow$. Proposition \ref{prop:source-coding} then gives $Y\in2\uparrow$. By \eqref{eq:source-size}, $X\gs4$ and
\[
 bY\ls X(X+1)<2^{X+1}.
\]
(The last inequality follows by induction starting at $X=4$.) Hence $bY\mid2^{X+1}$, and \eqref{eq:merged-congruence} yields $Y\mid\binom{2X}{X}$. Conversely, if $b\in2\uparrow$ and $Y\mid\binom{2X}{X}$, the same size bound gives $bY\mid2^{X+1}$ and \eqref{eq:merged-congruence} proves the divisibility on the left of \eqref{eq:merged-arithmetic}.
\end{proof}

Fix the source polynomials of Proposition \ref{prop:source-coding}. In the following, $a\in\N$ is the input, and $f,g,h,k,l\in\Z$ are the five unknowns of the new bridge. All the following expressions, except the explicitly introduced real quantities and formal indeterminates, are integer polynomials in these variables. Set
\begin{align}
 A_0&=2bYk,\qquad A=A_0^2,\label{eq:A-def}\\
 M_0&=2X+b+20,\qquad T_0=2M_0+2,\notag\\
 D_0&=(T_0^2-4)M_0^2A_0^2+4,\notag\\
 B&=X+1,\qquad C=B+(A-2)h,\qquad L=bYl.\label{eq:bridge-def}
\end{align}
In particular, if $k\ne0$, then $A\gs4$ is even and $bY\mid A$.

\begin{lemma}\label{lem:size}
Suppose that $f,k\ne0$ and $0\ls g<2\Cscr$. If $D_0\in\Sq(\Z)$, then
\begin{equation}\label{eq:size-bound}
 A>1000(b+2)4^X.
\end{equation}
For any fixed $a,f,g$ with $f\ne0$ and $0\ls g<2\Cscr$, there exist arbitrarily large positive integers $k$ with $D_0\in\Sq(\Z)$.
\end{lemma}
\begin{proof}
By \eqref{eq:source-size}, $X\gs3b>0$. In particular $M_0>2$ and $T_0$ is even. Lemma \ref{lem:pell-lucas} gives $M_0\abs{A_0}=u_j(T_0)$ for some $j>0$. Using $T_0-2=2M_0$ and \eqref{eq:index-cong}, we get $M_0\mid j$, and hence $j\gs M_0$. By \eqref{eq:growth},
\[
 |A_0|\gs\f{(2M_0+1)^{M_0-1}}{M_0}
 >2^{M_0-1}=4^X2^{b+19}>1000(b+2)4^X.
\]
The last inequality follows from $2^b\gs b+1$. Since $A=A_0^2$ and $|A_0|>1$, this proves \eqref{eq:size-bound}.

For the converse existence assertion, take $j$ to be any positive multiple of the order of $M_{T_0}$ modulo $2M_0bY$. Then \eqref{eq:matrix} gives $2M_0bY\mid u_j(T_0)$. Set $k=u_j(T_0)/(2M_0bY)$. It is a positive integer, and $M_0A_0=u_j(T_0)$ makes $D_0$ a square. As $j$ increases through such multiples, $k$ tends to infinity.
\end{proof}

We next encode the required real approximation and the range of $g$ by a single strict polynomial inequality. Put
\begin{equation}\label{eq:range}
 R_g=(g+1)(2\Cscr-g),\quad \kappa=(2\Cscr+1)^2,\quad
 S=2A-5,\quad J_0=bACS.
\end{equation}
For the formal indeterminate $t$, define
\begin{align}
 s(t)&=(A-4)t-1,\qquad H(t)=SL-2AC(2-A+t),\label{eq:norm-components}\\
 \mathcal R(t)&=4\kappa^2s(t)H(t)^2J_0^2-\l(\kappa(s(t)H(t)^2+J_0^2)-s(t)S^2R_g\r)^2.\label{eq:norm-poly}
\end{align}
Let $r_0+r_1t$ be the remainder of $\mathcal R(t)$ upon division by the monic polynomial $t^2-At+1$. Define
\begin{equation}\label{eq:O0}
 O_0=-f^2k^2(r_0^2+Ar_0r_1+r_1^2).
\end{equation}
Since the divisor used to compute the remainder is monic, both $r_0,r_1$, and hence $O_0$, are integer polynomials in $a,f,g,h,k,l$. There is no additional indeterminate $t$ in \eqref{eq:O0}.

\begin{lemma}\label{lem:norm}
If $O_0>0$, then $f,k\ne0$, $0\ls g<2\Cscr$, and for at least one $\ve\in\{\pm1\}$ we have
\begin{equation}\label{eq:near}
 |L-\al_\ve|^2<\Delta:=\f{R_g}{\kappa}\ls\f14,
\end{equation}
where $\al_\ve$ is defined by \eqref{eq:alpha} with the actual value of $C$ (not necessarily a Lucas term). Conversely, if $f,k\ne0$, $0\ls g<2\Cscr$, $C\gs1$, and \eqref{eq:near} holds for one sign, then $O_0>0$.
\end{lemma}
\begin{proof}
For any $g\in\Z$, we have
\[
 \Delta>0\iff0\ls g<2\Cscr\ \ \t{and}\ \ 
 4(g+1)(2\Cscr-g)\ls(2\Cscr+1)^2.
\]
If $O_0>0$, then $f,k\ne0$. With these two nonvanishing conditions, $b>1$, $Y>0$ and \eqref{eq:A-def} ensure $A\gs64$. Thus $\la=(A+\sqrt{A^2-4})/2>4$ and $S>0$. For either root $t$ of the quadratic equation $t^2-At+1=0$, we have
\[
 s(t)=t(t-4)\ \ \t{and}\ \ 
 O_0=-f^2k^2\mathcal R(\la)\mathcal R(\la^{-1}).
\]
Since $s(\la^{-1})<0$, both terms on the right of \eqref{eq:norm-poly} are nonpositive at $t=\la^{-1}$. Therefore $O_0>0$ forces $\mathcal R(\la)>0$.

Write $s=s(\la)>0$, $H=H(\la)$, and $D_*=J_0/\sqrt{s}$. Direct expansion yields
\begin{align}
 \mathcal R(\la)=-\kappa^2s^2
 &\big((H-D_*)^2-S^2\Delta\big)
 \big((H+D_*)^2-S^2\Delta\big).\label{eq:norm-factor}
\end{align}
If the product on the right before the minus sign is negative, then $\Delta>0$, and at least one of the squared differences is smaller than $S^2\Delta$. This already recovers $0\ls g<2\Cscr$, without assuming a bound on $X$ or a correct Lucas index. In view of the identity
\[
 \f1{\la-2}=\f{2-\la^{-1}}{2A-5}
\]
and \eqref{eq:norm-components}, we have $H-\ve D_*=S(L-\al_\ve)$. So \eqref{eq:near} holds.

Conversely, assume the stated hypotheses. Then $\Delta>0$ and $J_0>0$. The two centers $\al_+$ and $\al_-$ have distance $2bAC/\sqrt{s}>2bC\gs2$, whereas $\sqrt{\Delta}\ls1/2$. Thus \eqref{eq:near} makes exactly one factor in \eqref{eq:norm-factor} negative and the other positive. We get $\mathcal R(\la)>0$. At $\la^{-1}$, if $H(\la^{-1})\ne0$, the first term of \eqref{eq:norm-poly} is strictly negative. If $H(\la^{-1})=0$, the squared bracket in \eqref{eq:norm-poly} is strictly positive before squaring, since $J_0^2>0$, $R_g>0$ and $s(\la^{-1})<0$. In either case we have $\mathcal R(\la^{-1})<0$, so $O_0>0$.
\end{proof}

We can now give the new bridge in precisely the form required for both parts of Theorem \ref{Th1.1}.

\begin{proposition}\label{prop:bridge}
With the above definitions, we have
\begin{equation*}
 a\in\Acal\iff\exists f,g,h,k,l\,[D_0\in\Sq(\Z)\ \land\ O_0>0\ \land\ C=u_B(A)].
\end{equation*}
Furthermore, if $D_0\in\Sq(\Z)$ and $O_0>0$, then $A,B,C$ satisfy \eqref{eq:reverse-range}, regardless of whether $C=u_B(A)$.
\end{proposition}
\begin{proof}
Suppose first that $a\in\Acal$. Choose $f,g$ as in Proposition \ref{prop:source-coding}(i), so $\Delta>0$. By Lemma \ref{lem:size}, we may choose an arbitrarily large positive integer $k$ with $D_0\in\Sq(\Z)$. Choose $A$ large enough that the right-hand side of \eqref{eq:approx} is less than $\sqrt{\Delta}$. Put $C=u_{X+1}(A)$. By \eqref{eq:index-cong}, we can choose an integer $h$ satisfying \eqref{eq:bridge-def}. By Lemma \ref{lem:merged-arithmetic}, $bY\mid P_X^{(+1)}(A)$. Set $l=P_X^{(+1)}(A)/(bY)\in\Z$. We have $C>0$ and \eqref{eq:near} by Lemma \ref{lem:single-lucas}. Thus $O_0>0$ by Lemma \ref{lem:norm}, which proves the positive implication.

Conversely, suppose $D_0\in\Sq(\Z)$ and $O_0>0$. Lemma \ref{lem:norm} gives $f,k\ne0$ and $0\ls g<2\Cscr$. Hence Proposition \ref{prop:source-coding}(iii) and Lemma \ref{lem:size} apply. In particular, $A$ is even, $A\gs6$, $B=X+1$ is odd, and $1<B<A/2-1$. By \eqref{eq:bridge-def}, $A-2\mid C-B$. This proves the additional assertion of the proposition before we use the specified Lucas term.

If also $C=u_B(A)$, then $\la>A/2$ and \eqref{eq:size-bound}, \eqref{eq:approx} give
\[
 |\al_\ve-P_X^{(\ve)}(A)|<\f{24(b+2)4^X}{A}<\f1{16}
\]
for either sign. By \eqref{eq:near}, for some $\ve$ we have
\[
 |L-P_X^{(\ve)}(A)|<\f12+\f1{16}<1.
\]
As both quantities are integers, $L=P_X^{(\ve)}(A)$. Since $L=bYl$, Lemma \ref{lem:merged-arithmetic} gives $b\in2\uparrow$ and $Y\mid\binom{2X}{X}$. We have verified all hypotheses of Proposition \ref{prop:source-coding}(ii), which yields $a\in\Acal$.
\end{proof}

\section{Proof of Theorem \ref{Th1.1}}\label{sec:final}
\setcounter{equation}{0}

We continue to use the integer polynomials in Section \ref{sec:bridge}. For $A,B,C,\xi,y$ as in Section \ref{sec:modified}, denote the product of the three expressions in \eqref{eq:D}, \eqref{eq:F} and \eqref{eq:Istar} by
\[
 \mathcal D(\xi,y)=DFI^*.
\]
Here $A,B,C$ are always the polynomials in \eqref{eq:A-def} and \eqref{eq:bridge-def}. In particular, $\mathcal D$ is an integer polynomial in $a,f,g,h,k,\xi,y$.

\begin{proof}[Proof of Theorem \ref{Th1.1}(i)]
Define
\begin{equation}\label{eq:final-P}
 P_{\Acal}(a,f,g,h,k,l,x,y,m)
 =\mathcal M_2(D_0,\mathcal D(x,y),x^2O_0,m+1)
\end{equation}
with $\mathcal M_2$ given by \eqref{eq:Mk}. More explicitly, with
\begin{gather*}
 V_1=D_0,\quad V_2=\mathcal D(x,y),\quad U=1+V_1^2+V_2^2,\\
 \ga=2x^2O_0-1,\qquad \eta=m+1-\ga U^2,
\end{gather*}
the right-hand side of \eqref{eq:final-P} is
\begin{equation*}
 [\eta^2-\ga^2(V_1+U^2V_2)]^2-4\ga^4U^2V_1V_2.
\end{equation*}
Every expression is a polynomial with integer coefficients, and the only eight unknowns are $f,g,h,k,l,x,y,m$.

Suppose that $a\in\Acal$. By Proposition \ref{prop:bridge}, choose integers $f,g,h,k,l$ with $D_0\in\Sq(\Z)$, $O_0>0$ and $C=u_B(A)$. By Theorem \ref{thm:full-interpolation} with $N=1$, we can choose $x>0$ and $y\in\Z$ such that $\mathcal D(x,y)\in\Sq(\Z)$. Since $x^2O_0>0$, equivalence \eqref{eq:M-equivalence} provides a positive integer in the last argument of $\mathcal M_2$. Writing that integer as $m+1$ gives $P_{\Acal}=0$ with $m\in\N$.

Conversely, suppose that $P_{\Acal}=0$ with $f,g,h,k,l,x,y\in\Z$ and $m\in\N$. In view of Lemma \ref{lem:relation-combining}, we have
\[
 D_0,\mathcal D(x,y)\in\Sq(\Z)\ \ \t{and}\ \  x^2O_0>0.
\]
In particular, $x\ne0$ and $O_0>0$. By the last assertion of Proposition \ref{prop:bridge}, all conditions in \eqref{eq:reverse-range} hold. Theorem \ref{thm:modified-reverse} with $\xi=x$ therefore yields $C=u_B(A)$. Now Proposition \ref{prop:bridge} gives $a\in\Acal$. This proves \eqref{eq:mixed-main}.
\end{proof}

\begin{proof}[Proof of Theorem \ref{Th1.1}(ii)]
We share the positivity witness with the first Lucas witness by setting
\begin{equation*}
 \xi=zx\ \ \t{and}\ \ D_{\rm pos}=(3O_0-4)z^2+1.
\end{equation*}
Define the integer polynomial
\begin{equation}\label{eq:final-Q}
 \begin{split}
 Q_{\Acal}(a,f,g,h,k,l,x,y,z,m)
 &=H_3(D_0,\mathcal D(zx,y),D_{\rm pos},zx,\Pi(y),m),
 \end{split}
\end{equation}
where $H_3$ and $\Pi$ are defined by \eqref{eq:Hk} and \eqref{eq:Pi}, respectively. The only unknowns are $f,g,h,k,l,x,y,z,m$, all of which range over $\Z$.

\smallskip
\noindent\emph{Positive direction.}
Assume $a\in\Acal$. By Proposition \ref{prop:bridge}, choose $f,g,h,k,l$ such that $D_0\in\Sq(\Z)$, $O_0>0$ and $C=u_B(A)$. Since $O_0$ is now fixed, Lemma \ref{lem:positivity} supplies a positive integer $z$ such that $D_{\rm pos}\in\Sq(\Z)$. Apply Theorem \ref{thm:full-interpolation} with $N=z$. It yields a positive integer $\xi$ divisible by $z$ for which every residue class of $y$ modulo $\xi$ is possible with $\mathcal D(\xi,y)\in\Sq(\Z)$. Put $x=\xi/z$. By Lemma \ref{lem:universal-root}, choose $c\in\Z$ with $\xi\mid\Pi(c)$. The interpolation theorem gives a single $y$ such that
\[
 \mathcal D(\xi,y)\in\Sq(\Z)\ \ \t{and}\ \ y\equiv c\pmod\xi.
\]
Thus $zx=\xi\mid\Pi(y)$. The three square conditions and the one divisibility condition occurring in \eqref{eq:final-Q} now hold simultaneously. By Lemma \ref{lem:relation-combining}, there exists an integer $m$ with $Q_{\Acal}=0$.

\smallskip
\noindent\emph{Reverse direction.}
Suppose that $Q_{\Acal}=0$ with all nine unknowns integers. If $zx=0$, then by \eqref{eq:H-zero} we have $Q_{\Acal}=\Pi(y)^8\ne0$ (Lemma \ref{lem:universal-root}), a contradiction. Hence
\[
 \xi=zx\ne0,\qquad x,z\ne0.
\]
By Lemma \ref{lem:relation-combining}, we obtain
\[
 D_0,\mathcal D(zx,y),D_{\rm pos}\in\Sq(\Z).
\]
As $z\ne0$, Lemma \ref{lem:positivity} gives $O_0>0$. The final assertion of Proposition \ref{prop:bridge} again gives \eqref{eq:reverse-range}, and the nonzero witness $\xi$ and the square condition for $\mathcal D(\xi,y)$ allow us to apply Theorem \ref{thm:modified-reverse}. Consequently $C=u_B(A)$. By Proposition \ref{prop:bridge}, $a\in\Acal$.

This proves \eqref{eq:integer-main} and completes the proof of Theorem \ref{Th1.1}.
\end{proof}

\section{A quartic universal pair}\label{sec:quartic}
\setcounter{equation}{0}

In this section we prove Theorem \ref{Th1.2}. We use the digit-coding method of \cite[Sections~3 and~5]{J82}, but replace its central-binomial-coefficient bridge by a quadratic system based on the generating function of the Catalan numbers. To reduce the number of auxiliary products, we use a quadratic source system, a common digit mask, and a compact packing of the tests. All quantities in this section are defined afresh. Fix a program index $a_*\in\N$ and consider membership in $W_{a_*}$. The input $x$ is in $\N$ and the counted unknowns are positive integers. The index parameters, to be computed from $a_*$, are not counted in the number of unknowns or in the degree.

\medskip\noindent\textbf{The source system and the digit polynomials.}
By Matiyasevich's theorem \cite{M70} applied to the r.e. relation $x\in W_a$, there is a fixed integer polynomial $P(a,x,\mathbf t)$ such that $x\in W_a$ if and only if $P(a,x,\mathbf t)=0$ for some nonnegative integers $\mathbf t$. Introduce nonnegative intermediate variables for the monomial products in this fixed polynomial, leaving the final linear combination as an equation. Specializing $a=a_*$ gives a system of equations
\begin{equation}\label{eq:quartic-source}
 p_i(x,z_1,\ldots,z_\rho)=0\quad(1\ls i\ls t)
\end{equation}
with $\deg p_i\ls2$ and $z_j\in\N$, which has a solution if and only if $x\in W_{a_*}$. Here $\rho,t\gs1$ and the list of monomials available for each $p_i$ are fixed independently of $a_*$. The coefficients of the $p_i$ are integer polynomials in $a_*$, obtained by specialization of the fixed quadratic system. In particular they are total computable functions of $a_*$. Positive source witnesses, if used initially, can be replaced by $z_j+1$ before this construction.

Put $e_j=5^j$ for $0\ls j\ls\rho+t$, and let $d=e_{\rho+t}$ and $h=8d$. The symbols $d$ and $h$ thus denote fixed positive integers. Write
\[
 p_i(x,\mathbf z)=\sum_{|\mathbf a|\ls2}p_{i,\mathbf a}\,x^{a_0}\prod_{j=1}^{\rho}z_j^{a_j}
 \ \ \t{and}\ \ 
 m_{\mathbf a}=\f{2!}{(2-|\mathbf a|)!\prod_{j=0}^{\rho}a_j!},
\]
where $\mathbf a=(a_0,\ldots,a_\rho)\in\N^{\rho+1}$ and $|\mathbf a|=\sum_{j=0}^\rho a_j$. In the formal variable $Z$, define the following fixed digit polynomials:
\begin{gather}
 \ell_*(Z)=\sum_{j=1}^{\rho+t} Z^{e_j},\ \ 
 \lambda_*(Z)=\sum_{j=0}^{h-1}Z^j,\notag\\
 D_*(Z)=Z^{4d}+\ell_*(Z)+\sum_{i=1}^t\sum_{|\mathbf a|\ls2}
 \f{2p_{i,\mathbf a}}{m_{\mathbf a}}
 Z^{e_{\rho+i}-\sum_{j=0}^\rho a_je_j}.\label{eq:quartic-Dstar}
\end{gather}
The exponents in the last sum are nonnegative and at most $d$, since $e_{\rho+1}>2e_\rho$. All the coefficients are integers, as $m_{\mathbf a}$ is either 1 or 2. Let $z$ be the least power of two greater than the absolute values of the coefficients of $D_*$ and greater than $\rho+t+2$. Define the index data by
\begin{equation}\begin{gathered}
 q=2z,\ K_0=q^h,\ \kappa=32q^{4h+1},\
 E_*(Z)=z\lambda_*(Z)+D_*(Z),
 \\u=\ell_*(q),\ y\ =E_*(q),\  d_0=y-z\lambda_*(q)=D_*(q).
\end{gathered}\label{eq:quartic-index}
\end{equation}
Thus $0<u,y<K_0$, because the coefficients of $E_*$ belong to $[1,q-1]$. Also $D_*(B)>0$ for $B\gs q$: the leading term $B^{4d}$ dominates the sum of the absolute values of the lower terms, which is at most $z\sum_{j=0}^d B^j$. In particular, $d_0>0$. The parameters of the universal polynomial will be $z,u,y$. The preceding construction defines a total computable map $a_*\mapsto(z,u,y)$: the source coefficients are computed by specialization, the least power of two is found by a finite search, and $u,y$ are then obtained by evaluation of the digit polynomials. The other index data in \eqref{eq:quartic-index} are integer polynomials in these parameters with fixed exponents. No source coefficient $p_{i,\mathbf a}$ occurs in the final quadratic equations except through these index data.

For a movable positive integer bound $b>x$, use the base $B=\kappa b^2$. We call
\begin{equation}\label{eq:quartic-code}
 c=1+xB+\sum_{j=1}^{\rho+t}z_jB^{e_j}\ \ \t{with}\ z_j\in\{0,\ldots,b-1\},
\end{equation}
a legal code. The last $t$ digits are additional digits used to mark the equation positions. They need not be forced to zero.

\begin{lemma}\label{lem:quartic-mask}
For a legal code, put $c_*(Z)=1+xZ+\sum_{j=1}^{\rho+t}z_j Z^{e_j}$ and write
\[
 c_*(Z)^2D_*(Z)=\sum_{j=0}^{h-1} a_j Z^j.
\]
Then
\begin{equation}\label{eq:quartic-tagcoeff}
 a_{e_j}=1\ (1\ls j\ls\rho),\qquad
 a_{e_{\rho+i}}=1+2p_i(x,z_1,\ldots,z_\rho)\ (1\ls i\ls t),
\end{equation}
and $|a_j|<B/2$ for every $j$. If $b$ is a power of two and $\tau_2$ denotes the number of binary carries in addition, then
\begin{equation}\label{eq:quartic-borrowtest}
 \tau_2(c^2D_*(B),(B-2)\ell_*(B))=0
 \iff p_i(x,z_1,\ldots,z_\rho)=0\quad(1\ls i\ls t).
\end{equation}
\end{lemma}
\begin{proof}
A term from the last sum of \eqref{eq:quartic-Dstar} contributes to the coefficient at $e_s$ only if
\[
 e_{\rho+i}+\sum_{j=0}^{\rho+t} b_j e_j
 =e_s+\sum_{j=0}^{\rho}a'_j e_j,
\]
where both $\sum b_j$ and $\sum a'_j$ are at most 2. The two sides are sums of at most three powers of 5, so uniqueness of base-5 digits forces $s=\rho+i$ and equality of the two multi-indices after extending the shorter one by zeros. The coefficient contributed by these terms is $2p_i$. Each monomial in $\ell_*(Z)$ contributes exactly 1 to the coefficient at its own position and nothing at the other marked positions, by the same uniqueness argument (or by setting the subtracted multi-index equal to zero). The term $Z^{4d}$ contributes nothing at the marked positions. This proves \eqref{eq:quartic-tagcoeff}. The product has degree at most $6d<h$. Each coefficient is bounded in absolute value by $z c_*(1)^2\ls z(\rho+t+2)^2 b^2<B/2$.

For the last assertion, convert the signed coefficients $a_j$ to ordinary base-$B$ digits, starting at position 0. The incoming carry at position 0 is 0. If the incoming carry $\varepsilon_j$ is 0 or $-1$, then $-B<a_j+\varepsilon_j<B$, so the outgoing carry is again 0 or $-1$. The digit at that position is the least nonnegative residue of $a_j+\varepsilon_j$ modulo $B$. At a marked position, \eqref{eq:quartic-tagcoeff} gives an odd coefficient. For an integer $p$ with $|1+2p|<B/2$, the least residue of $1+2p+\varepsilon_j$ lies in $\{0,1\}$ if and only if $p=0$, for either value of $\varepsilon_j$. Since $B$ is a power of two, the binary mask $B-2$ forbids precisely the digits outside $\{0,1\}$ at such a position. The mask is zero at unmarked positions. Moreover $c^2D_*(B)>0$, and the coefficient bounds give $c^2D_*(B)<Q$ for $Q=B^h$. Thus the digit conversion uses the ordinary nonnegative expansion, without a negative final carry. This proves \eqref{eq:quartic-borrowtest}. (The equivalence of carry-freeness and disjoint binary ones follows immediately from binary addition.)
\end{proof}

\medskip\noindent\textbf{Extracting a Catalan coefficient by a quadratic radicand.}
For $A>1$ and $j\in\N$ write
\[
 \psi_A(j)=u_j(2A),\qquad \chi_A(j)=v_j(2A)/2.
\]
By Lemma \ref{lem:pell-lucas}, the positive solutions of $D^2-(A^2-1)C^2=1$ are $(D,C)=(\chi_A(j),\psi_A(j))$ with $j\in\Z^+$. We use \eqref{eq:growth}, \eqref{eq:index-cong} and the identity $\chi_A(j)=A\psi_A(j)-\psi_A(j-1)$, which follows from \eqref{eq:binet}. Let $\Cat(j)=\binom{2j}{j}/(j+1)$ for $j\in\N$ be the Catalan numbers.

\begin{lemma}\label{lem:quartic-root}
Let $n\gs3$ and $A>1000\cdot4^n$ be integers. Put
\[
 C=\psi_A(n),\quad D=\chi_A(n),\quad J=AC,\quad P=J-D=\psi_A(n-1),\quad K_p=AP.
\]
Then $M=J(J-4K_p)>0$, and $L=\lceil\sqrt M\rceil$ satisfies
\begin{equation}\label{eq:quartic-ceil}
 L-1<\sqrt M<L,\qquad L\equiv-\Cat(n-1)\pmod A.
\end{equation}
\end{lemma}
\begin{proof}
Put $a=2A$, $\theta=A+\sqrt{A^2-1}$ and $\omega=a/(\theta-\theta^{-1})=(1+\theta^{-2})/(1-\theta^{-2})$. The generating function
\[
 f_*(t)=\sqrt{1-4t}=\sum_{j=0}^\infty c_jt^j\quad(|t|<1/4)
\]
has coefficients $c_0=1$ and $c_j=-2\Cat(j-1)$ for $j\gs1$, with $|c_j|\ls4^j$. For fixed $n$ define the polynomial
\[
 E_n(a)=c_n+a\sum_{j=1}^n c_{n-j}u_j(a).
\]
Its half is an integer at $a=2A$ and is congruent to $-\Cat(n-1)$ modulo $A$. We show that this integer is $L$.

Let $\alpha=aC\sqrt{1-4/\theta}$. By \eqref{eq:binet},
\[
 \alpha=\omega(\theta^n-\theta^{-n})f_*(1/\theta),\qquad
 au_j(a)=\omega(\theta^j-\theta^{-j}).
\]
Subtracting gives
\begin{align*}
 E_n(a)-\alpha={}&(\omega-1)|c_n|+\omega\sum_{i>n}|c_i|\theta^{n-i}\\
 &+\omega\sum_{j=1}^{n-1}|c_{n-j}|\theta^{-j}
 -\omega\theta^{-n}(1-f_*(1/\theta)).
\end{align*}
Here $\theta>8$, $1<\omega<65/63$ and $\omega-1<1/\theta$. Since $|c_{n+1}|\gs2$, the positive tail alone minus the last term is greater than $1/\theta$. The two sums, before multiplication by $\omega$, are bounded above by $8\cdot4^n/\theta$ and $4^n/(4\theta-1)$, respectively. Thus
\begin{equation}\label{eq:quartic-error}
 \f1\theta<E_n(a)-\alpha<\f{12\cdot4^n}{\theta}.
\end{equation}
Since $P=C/\theta-\theta^{-n}$, we have $0<P/C<1/\theta<1/4$. Therefore $M>0$. For $\beta=2\sqrt M=aC\sqrt{1-4P/C}$, rationalizing the difference gives
\[
 0<\beta-\alpha=
 \f{4a\theta^{-n}}{\sqrt{1-4P/C}+\sqrt{1-4/\theta}}
 <3\theta^{1-n}<1/\theta.
\]
For the penultimate inequality use $a=\theta+\theta^{-1}$ and the denominator $>2\sqrt{1-4/\theta}>\sqrt{2}$, with $\theta>8$. Combining with \eqref{eq:quartic-error} yields $0<E_n(a)-2\sqrt M<2$ by the assumed bound on $A$. This proves \eqref{eq:quartic-ceil}.
\end{proof}

For later use, write $\sigma_2(j)$ for the sum of the binary digits of $j$. Kummer's theorem (cf. \cite[Lemmas~2.14--2.16]{J82}) gives
\begin{equation}\label{eq:quartic-catval}
 \nu_2(\Cat(2R))=\nu_2\binom{4R}{2R}=\sigma_2(2R)=\sigma_2(R)=\nu_2\binom{2R}{R}\quad(R\gs1),
\end{equation}
since $2R+1$ is odd. This explains why a square-root generating function suffices for the central-binomial-coefficient bridge in a quadratic system.

\medskip\noindent\textbf{The quadratic system.}
We give the complete system, to make both the variable count and the degree transparent. The five disjoint lists of positive-integer unknowns are
\begin{align*}
 &B,C_1,D_1,c,D_0,\ell,\mu,\alpha,\Delta,\lambda,\ve,S,T,R,\\
 &L_B,\Gamma,J_1,c_2,Q_2,N,H_c,H_\ell,N_2 &&(23\text{ source}),\\
 &A,\phi,D,J_{AC},F,G,H_J,I,i,j,o,X_c,Y_0,J_{GH} &&(14),\\
 &V,k,D',J' &&(4),\\
 &w,r &&(2),\\
 &K_p,L,\eta,s &&(4).
\end{align*}
The following are affine expressions, not additional unknowns:
\begin{gather*}
 b=x+\ve,\quad Q=1+L_B-\lambda,\quad e=z\lambda+D_0,\quad C=2R+1+C_1+\phi,\\
 W_*=(1+2c_2)(d_0+K_0u)+q(u+c_2).
\end{gather*}
Also $g=c-1-xB$ is a quadratic expression, not an unknown. The symbols $z,u,y,h,q,K_0,\kappa,d_0$ are the index data of \eqref{eq:quartic-index}. The independent source equations are
\begin{gather}
 L_B=\lambda B,\quad B=\kappa b^2,\quad
 \Gamma=2AB-B^2-1,\quad J_1=AC_1,\label{eq:quartic-base}\\
 D_1=Q+J_1-C_1B+\alpha\Gamma,\quad D_1^2=J_1^2-C_1^2+1,\label{eq:quartic-firstpell}\\
 C_1=h+\Delta(A-1),\quad c_2=c^2,\label{eq:quartic-firstindex}\\
 \begin{split}
 Q-(1+2c_2)(D_0+K_0\ell)-B(\ell+c_2)
 &=K_0-W_*+(W_*+\mu)(B-q),
 \end{split}\label{eq:quartic-complement}\\
 Q_2=Q^2,\quad N=Q_2^2,\quad QH_c=Nc_2,\quad QH_\ell=N\ell,\label{eq:quartic-shifts}\\
 S=g+Q\ell+Q_2e+H_cD_0,\label{eq:quartic-S}\\
 \begin{split}
 T={}Q-1-(b-1)\ell+(Q-1+(1-q)\lambda)(Q+Q_2)+(B-2)H_\ell,
 \end{split}\label{eq:quartic-T}\\
 N_2=N^2,\quad R=S(N_2-N)+(T+1)(N_2-1).\label{eq:quartic-R}
\end{gather}
The ten equations for the main Lucas term are
\begin{gather}
 J_{AC}=AC,\quad D^2=J_{AC}^2-C^2+1,\label{eq:quartic-mainpell}\\
 X_c=2iC,\quad Y_0=DX_c,\quad F^2=Y_0^2-X_c^2+1,\label{eq:quartic-firstwitness}\\
 G=A+F(F-A),\quad H_J=2R+1+jC,\label{eq:quartic-secondbase}\\
 J_{GH}=GH_J,\quad I^2=J_{GH}^2-H_J^2+1,\quad I=D+oF.\label{eq:quartic-secondwitness}
\end{gather}
Finally, the eight size, power and extraction equations are
\begin{gather}
 V=R(w+k),\quad A=N_2V,\quad J'=(R+1)A,\quad D'^2=J'^2-A^2+1,\label{eq:quartic-size}\\
 D=bw+J_{AC}-2C+r(4A-5),\label{eq:quartic-power}\\
 K_p=A(J_{AC}-D),\quad \eta=L^2-J_{AC}^2+4J_{AC}K_p,\label{eq:quartic-radicand}\\
 (4L-2\eta-3)L=N_2s.\label{eq:quartic-extract}
\end{gather}
The use of the quadratic product identities in \eqref{eq:quartic-shifts}, rather than direct cubing of $Q$, is important for the degree bound. In particular, the bound $N=Q^4$ needs only the one intermediate square $Q_2$. All 35 substantive equations have degree at most two, including their dependence on the input $x$, after the displayed derived expressions are substituted. In \eqref{eq:quartic-S} the only occurrence of $g$ is linear, so its quadratic definition does not increase the degree beyond two. We now prove that this system defines precisely $W_{a_*}$ for these index data.

\medskip\noindent\textbf{Reverse direction: bounds before power recovery.}
Suppose all 47 unknowns are positive integers satisfying the system. By \eqref{eq:quartic-base}, $Q=1+\lambda(B-1)\gs B$, $B=\kappa b^2>4b$, and $B>q+1$. The right-hand side of \eqref{eq:quartic-complement} is positive since it equals
\[
 K_0+W_*(B-q-1)+\mu(B-q).
\]
Consequently
\begin{equation}\label{eq:quartic-prebounds}
 (1+2c^2)(D_0+K_0\ell)+B(\ell+c^2)<Q.
\end{equation}
In particular, $Bc^2<Q$, $B\ell<Q$, $D_0<Q/3$, and $c^2D_0<Q$. The identity for $Q$ gives $z\lambda<Q/2$, so $0<e<Q$. Moreover, $QH_c=Nc^2$ and $QH_\ell=N\ell$ give $H_c=Q^3c^2$ and $H_\ell=Q^3\ell$. Set
\begin{align*}
 S_1&=g,&T_1&=Q-1-(b-1)\ell,\\
 S_2&=\ell+Qe,&T_2&=(B-q)\lambda(1+Q),\\
 S_3&=c^2D_0,&T_3&=(B-2)\ell.
\end{align*}
Equations \eqref{eq:quartic-S} and \eqref{eq:quartic-T} become
\begin{equation}\label{eq:quartic-blocks}
 S=S_1+QS_2+Q^3S_3,\quad T=T_1+QT_2+Q^3T_3,\quad N=Q^4.
\end{equation}
Even before we know $b\in2\uparrow$, we have $Q/2<T_1<Q$, $0<S_2,T_2<Q^2$ and $0<S_3,T_3<Q$. For $T_1$, use $(b-1)\ell<Q/4$ by \eqref{eq:quartic-prebounds} and $Q>4$. For $T_2$, use $(B-q)\lambda<Q-1$. If $g\gs0$ then $g<c<Q$; if $g<0$ it only decreases the upper bound on $S$. In either case $S<N$. Also $g\gs-xB>-Q^2$ and $QS_2>Q^2$ because $x<b<B\ls Q$ and $e,\ell>0$. Thus $0<S,T<N$. It follows from \eqref{eq:quartic-R} that $R>N>64$.

By \eqref{eq:quartic-size}, $R\mid A$ and $A=\psi_{R+1}(m)$ for some $m\in\Z^+$. The index congruence modulo $R$ gives $R\mid m$. Consequently,
\begin{equation}\label{eq:quartic-bigA}
 A\gs(2R+1)^{R-1}>1000\cdot4^{2R+1}\ \ \t{and}\ \ A>N^N.
\end{equation}
Indeed $2R+1>2^7$, $7(R-1)>4R+12$, and $R-1\gs N$. Also $A=N^2R(w+k)>bw$, since $R>N\gs B>b$. All these size facts have been obtained without the power recovery or a carry-free condition.

\medskip\noindent\textbf{Reverse direction: verification of the Lucas certificate.}
Put $n=2R+1$. Equations \eqref{eq:quartic-mainpell} and \eqref{eq:quartic-firstwitness} give positive indices $a_1,a_2$ with
\[
 C=\psi_A(a_1),\quad D=\chi_A(a_1),\quad 2iC^2=\psi_A(a_2),\quad F=\chi_A(a_2).
\]
The second assertion follows from 
$$Y_0^2-X_c^2=(D^2-1)(2iC)^2=(A^2-1)(2iC^2)^2.$$ As $C>n$, we have $a_1\gs2$. Lemma \ref{lem:square-divisibility} gives $a_1C\mid a_2$, so $a_2>a_1$. We claim that $a_2/a_1$ is even. If $C$ is even this follows from $a_1C\mid a_2$. If $C$ is odd, the recurrence gives $a_1$ odd (since $\psi_A(j)\equiv j\pmod2$), and $2iC^2$ even gives $a_2$ even. Hence in both cases $a_2=2a_1v$ for some $v\in\Z^+$. Taking the $2v$-th power of the Pell unit of index $a_1$ and reducing its constant coefficient modulo $C$, we find
\[
 F\equiv \chi_A(a_1)^{2v}\equiv1\pmod C.
\]
The last congruence follows from $\chi_A(a_1)^2-(A^2-1)C^2=1$. In particular the base $G$ in \eqref{eq:quartic-secondbase} satisfies
\begin{equation}\label{eq:quartic-Gcong}
 G>1,\qquad G\equiv A\pmod F,\qquad G\equiv1\pmod C.
\end{equation}
For positivity use $F=\chi_A(a_2)>A$. This even-witness variant of the two-Pell certificate needs no auxiliary variable for $F^2$.

By \eqref{eq:quartic-secondwitness} we have $H_J=\psi_G(a_3)$ and $I=\chi_G(a_3)$ for a positive integer $a_3$. In view of $I=D+oF$ and \eqref{eq:quartic-Gcong},
\[
 \chi_A(a_3)\equiv\chi_A(a_1)\pmod{\chi_A(a_2)}.
\]
We recall the relevant index separation for companion terms. The identities
\[
 \chi_A(-j)=\chi_A(j),\qquad
 \chi_A(a_2+j)+\chi_A(a_2-j)=2\chi_A(a_2)\chi_A(j)
\]
show that the sequence is antiperiodic at $2a_2$ modulo $\chi_A(a_2)$. All residues are represented by $0$ and $\pm\chi_A(j)$ for $0\ls j<a_2$. These representatives are distinct because $2\chi_A(a_2-1)<\chi_A(a_2)$ (use the Binet formula and $A\gs2$, $a_2\gs2$). Thus the displayed congruence implies $a_3\equiv\pm a_1\pmod{a_2}$. On the other hand $H_J=n+jC$ and $G\equiv1\pmod C$ imply $a_3\equiv n\pmod C$, since $u_j(2)=j$. As $C\mid a_2$, we obtain $a_1\equiv\pm n\pmod C$. From \eqref{eq:quartic-bigA}, $A>n$, so $C\gs2A>2n$ and $a_1\ls C/2$ by \eqref{eq:growth}. Hence $a_1+n<C$, and the congruence forces $a_1=n$. We have proved
\begin{equation*}
 C=\psi_A(2R+1)\ \ \t{and}\ \ D=\chi_A(2R+1).
\end{equation*}

\medskip\noindent\textbf{Reverse direction: recovery of powers and central divisibility.}
For every integer $v$ and positive integer $j$, the recurrence gives
\begin{equation}\label{eq:quartic-expcong}
 v^j\equiv v\psi_A(j)-\psi_A(j-1)=\chi_A(j)-(A-v)\psi_A(j)\pmod{2Av-v^2-1},
\end{equation}
provided the modulus is nonzero. To see this, use $v^2\equiv2Av-1$ modulo that modulus and induct on $j$. Apply \eqref{eq:quartic-expcong} with $v=2$, $j=n$. By \eqref{eq:quartic-power}, $bw\equiv2^n\pmod{4A-5}$. Both positive numbers are less than $A$ by \eqref{eq:quartic-bigA} and $bw<A$. Thus $bw=2^n$, and $b$ and $B$ are powers of two.

The first Pell equations \eqref{eq:quartic-firstpell} give $C_1=\psi_A(a_4)$, $D_1=\chi_A(a_4)$ with $a_4<n$ since $C_1<C$. Also $a_4\equiv h\pmod{A-1}$ by \eqref{eq:quartic-firstindex}. We have $h<B\ls Q<N<R<n<A-1$ by the fixed base factor and \eqref{eq:quartic-bigA}. It follows that $a_4=h$. Use \eqref{eq:quartic-expcong} with $v=B$, $j=h$ in the congruence for $D_1$ in \eqref{eq:quartic-firstpell}. This gives $Q\equiv B^h\pmod\Gamma$. Here $Q<N<A$, $B^h<N^N<A$, and $\Gamma=2AB-B^2-1>A$ since $1<B<A$. Thus $Q=B^h$. In particular,
\begin{equation}\label{eq:quartic-correctpowers}
 \lambda=\lambda_*(B),\quad Q=B^h,\quad N=Q^4\in2\uparrow,\quad N_2=N^2\mid A.
\end{equation}

By Lemma \ref{lem:quartic-root}, $M=J_{AC}(J_{AC}-4K_p)>0$. Equations \eqref{eq:quartic-radicand} and \eqref{eq:quartic-extract} give $\eta=L^2-M>0$ and $4L-2\eta-3>0$. Thus $\eta\ls2L-2$, and $L-1<\sqrt M<L$. The multiplier in \eqref{eq:quartic-extract} is odd, hence $N^2\mid L$ by \eqref{eq:quartic-correctpowers}. Lemma \ref{lem:quartic-root}, $N^2\mid A$ and \eqref{eq:quartic-catval} give
\begin{equation}\label{eq:quartic-central}
 N^2\mid\Cat(2R)\iff N^2\mid\binom{2R}{R};
\end{equation}
in particular the central divisibility holds.

\medskip\noindent\textbf{Reverse direction: recovery of the source digits.}
The definition of $R$ in \eqref{eq:quartic-R} is equivalent to
\[
 R=(S+T)N^2+(N-1-S)N+N-1-T.
\]
As $0<S,T<N$ and $N$ is a power of two, the three displayed digit blocks give
\[
 \sigma_2(R)=\sigma_2(S+T)+2\log_2N-\sigma_2(S)-\sigma_2(T).
\]
Here $S+T$ may have an extra leading digit, which does not affect disjointness from the other two blocks. Since $\tau_2(S,T)=\sigma_2(S)+\sigma_2(T)-\sigma_2(S+T)$, Kummer's theorem and \eqref{eq:quartic-central} imply $\tau_2(S,T)=0$; this is also \cite[Lemma~2.16]{J82}.

First, $g$ cannot be negative. If $g<0$ then $-Q/2<g<0$ since $g\gs-xB$ and $Q=B^h>2xB$. The lowest $Q$-block of $S$ in \eqref{eq:quartic-blocks} is $Q+g>Q/2$. The corresponding block of $T$ is $T_1>Q/2$, giving a common highest bit below $Q$, a contradiction. Thus $0\ls g<Q$. The block widths $Q,Q^2,Q$ in \eqref{eq:quartic-blocks} are now powers of two; their disjoint binary positions imply that $\tau_2(S_i,T_i)=0$ for each $i=1,2,3$.

From $\tau_2(S_2,T_2)=0$, splitting into two $Q$-blocks and using $0<\ell,e<Q$, we see that all the $h$ base-$B$ digits of $\ell,e$ are less than $q$. Indeed $(B-q)\lambda=(B-q)\lambda_*(B)$ has digit $B-q$ in every position, and $B,q$ are powers of two with $B>q$. Write $\ell_j,e'_j$ for these digits, and set
\[
 \ell_q=\sum_{j=0}^{h-1}\ell_jq^j\ \ \t{and}\ \ e_q=\sum_{j=0}^{h-1}e'_jq^j.
\]
Both belong to $[0,K_0)$. The digit bounds just obtained imply that $T_1=Q-1-(b-1)\ell$ has base-$B$ digits $B-1-(b-1)\ell_j$, with no borrow, because $(b-1)(q-1)<B$. From $\tau_2(S_1,T_1)=0$, the base-$B$ digits $g_j$ of $g$ satisfy
\begin{equation}\label{eq:quartic-early-g}
 0\ls g_j\ls(b-1)\ell_j\ls (b-1)(q-1).
\end{equation}
Indeed the binary ones in $g_j$ must be among those of $(b-1)\ell_j$. This preliminary bound does not yet assume that the mask $\ell$ has its correct sparse form.

Define $c_q=1+xq+\sum_{j=0}^{h-1}g_jq^j$. Then $c_q\equiv c\pmod{B-q}$, and \eqref{eq:quartic-early-g} gives $0<c_q<2bK_0$ using $x<b$ and $K_0>q+1$. Reducing \eqref{eq:quartic-complement} modulo $B-q$, using \eqref{eq:quartic-correctpowers} and $D_0=e-z\lambda$, yields
\begin{equation*}
 (1+2c_q^2)(e_q-y)+\big(K_0(1+2c_q^2)+q\big)(\ell_q-u)\equiv0\pmod{B-q}.
\end{equation*}
The absolute value of the left-hand side is less than
\[
 (1+8b^2K_0^2)(K_0+K_0^2)+qK_0<20b^2K_0^4<B-q,
\]
where the last inequality follows from $B=32qK_0^4 b^2$. Thus the left-hand side is zero as an integer. Since $|e_q-y|<K_0$ and the coefficient of $\ell_q-u$ is strictly greater than $K_0(1+2c_q^2)$, this equality forces $\ell_q=u$, and then $e_q=y$. Uniqueness of base-$q$ digits gives
\begin{equation}\label{eq:quartic-recovery}
 \ell=\ell_*(B),\quad e=E_*(B),\quad D_0=D_*(B).
\end{equation}
The mask $T_1$ consequently allows digits of $g$ less than $b$ exactly at the marked positions and forces zero digits elsewhere. Hence $c$ is a legal code as in \eqref{eq:quartic-code}. Finally $\tau_2(S_3,T_3)=0$ is now the condition in Lemma \ref{lem:quartic-mask}. By that lemma all equations \eqref{eq:quartic-source} hold. This proves that $x\in W_{a_*}$.

\medskip\noindent\textbf{Positive direction.}
Suppose $x\in W_{a_*}$. Fix an actual solution of \eqref{eq:quartic-source}, assign zero to the $t$ tag digits, and choose a power of two $b>x$ exceeding all the source digits. We may increase $b$ arbitrarily. Set $B=\kappa b^2$, $Q=B^h$, $\lambda=\lambda_*(B)$, and take $\ell,e,D_0,c$ as in \eqref{eq:quartic-recovery} and \eqref{eq:quartic-code}. Define $c_2=c^2$ and $\ve=b-x>0$. All quantities in these definitions are positive.

To verify the existence of a positive complement lift, put
\[
 P_B=(1+2c_2)(D_0+K_0\ell)+B\ell,\qquad
 P_q=(1+2c_2)(d_0+K_0u)+qu.
\]
Here the same integer $c_2$ is used in both expressions, and $W_*=P_q+qc_2$. Equation \eqref{eq:quartic-complement} is satisfied by
\begin{equation}\label{eq:quartic-mu}
 \mu=\f{B^h-q^h-(P_B-P_q)}{B-q}-c_2-W_*.
\end{equation}
The numerator of the quotient is divisible by $B-q$ because $D_*(q)=d_0$ and $\ell_*(q)=u$. With the source digits fixed, $c$ is a polynomial in $B$ of degree at most $d$. Therefore $P_B-P_q$ is a polynomial in $B$ of degree at most $6d$ divisible by $B-q$; the quotient has degree at most $6d-1$. The remaining subtracted terms have degree at most $2d$. As $(B^h-q^h)/(B-q)$ is monic of degree $h-1=8d-1$, \eqref{eq:quartic-mu} gives $\mu>0$ for all sufficiently large $B=\kappa b^2$. Fix such a choice of $b$. This supplies all pre-recovery bounds by \eqref{eq:quartic-prebounds}.

Define $Q_2,N,H_c,H_\ell,N_2,S,T,R$ by \eqref{eq:quartic-shifts}--\eqref{eq:quartic-R}. All these quantities are positive. By the legal code and the digit polynomials, the first two block additions are carry-free. The third addition is carry-free by Lemma \ref{lem:quartic-mask}. Thus $\tau_2(S,T)=0$ by \eqref{eq:quartic-blocks}, and the digit identity for $R$ proved above gives $N^2\mid\binom{2R}{R}$.

With these source data fixed, put $n=2R+1$ and $w=2^n/b\in\Z^+$. Choose arbitrarily large positive Pell values $A=\psi_{R+1}(m)$ divisible by $RN^2$: choose $m$ to be a positive multiple of the order of the recurrence matrix of trace $2(R+1)$ modulo $RN^2$, using \eqref{eq:matrix}. For sufficiently large such $A$, $k=A/(RN^2)-w$ is positive. Put $V=R(w+k)$, $J'=(R+1)A$ and $D'=\chi_{R+1}(m)$ to obtain \eqref{eq:quartic-size}.

Set $C=\psi_A(n)$, $D=\chi_A(n)$, $C_1=\psi_A(h)$, $D_1=\chi_A(h)$, and put $J_{AC}=AC$, $J_1=AC_1$, $L_B=\lambda B$, $\Gamma=2AB-B^2-1$. The congruence and growth properties give positive integers $\Delta=(C_1-h)/(A-1)$ and $\phi=C-n-C_1$ for sufficiently large $A$ (here $h<n$). Congruence \eqref{eq:quartic-expcong} supplies the quotients $r$ and $\alpha$ in \eqref{eq:quartic-power} and \eqref{eq:quartic-firstpell}, respectively. They are positive for large $A$, since their numerators are $2\psi_A(n)-\psi_A(n-1)-2^n$ and $B\psi_A(h)-\psi_A(h-1)-B^h$, respectively, which tend to positive infinity with leading growth orders $A^{n-1}$ and $A^{h-1}$; the denominators $4A-5,\Gamma$ are positive. This also gives \eqref{eq:quartic-firstindex} and \eqref{eq:quartic-mainpell}.

For the main Lucas certificate choose an index $a_2$ which is a positive multiple of the order of the recurrence matrix of trace $2A$ modulo $2C^2$. Then $i=\psi_A(a_2)/(2C^2)$ is a positive integer. Put $$X_c=2iC,\ Y_0=DX_c,\ F=\chi_A(a_2),\ G=A+F(F-A).$$ The reasoning leading to \eqref{eq:quartic-Gcong} applies and proves these congruences for this choice. Choose $a_3$ arbitrarily large, congruent to $n$ modulo both the period of the recurrence matrix of trace $2A$ modulo $F$ and the integer $C$. Then $I=\chi_G(a_3)\equiv D\pmod F$ and $H_J=\psi_G(a_3)\equiv n\pmod C$. Consequently $o=(I-D)/F$ and $j=(H_J-n)/C$ are positive integers when $a_3$ is large enough. Putting $J_{GH}=GH_J$ verifies all of \eqref{eq:quartic-firstwitness}--\eqref{eq:quartic-secondwitness}.

Finally, put $K_p=A(J_{AC}-D)$ and $L=\lceil\sqrt{J_{AC}(J_{AC}-4K_p)}\rceil$. Lemma \ref{lem:quartic-root}, $N^2\mid A$, \eqref{eq:quartic-catval} and the central divisibility proved in the positive direction give $N^2\mid L$. The strict inequalities in \eqref{eq:quartic-ceil} give $$0<\eta=L^2-J_{AC}(J_{AC}-4K_p)\ls2L-2.$$ Hence $s=(4L-2\eta-3)L/N_2$ is a positive integer. All size and extraction equations now hold, and every one of the 47 unknowns is positive.

\begin{proof}[Proof of Theorem \ref{Th1.2}]
The preceding arguments prove both implications for the quadratic system \eqref{eq:quartic-base}--\eqref{eq:quartic-extract}. There are 23 source unknowns, 14 main Lucas unknowns, four size unknowns, two power unknowns and four extraction unknowns, giving $23+14+4+2+4=47$ in total. With $E_1,\ldots,E_{35}$ denoting the left-hand sides minus the right-hand sides of the 35 substantive equations, let
\[
 \mathcal U(x,z,u,y,\mathbf v)=\sum_{i=1}^{35}E_i(x,z,u,y,\mathbf v)^2\ \ \t{with}\ \ 
 \mathbf v=(v_1,\ldots,v_{47}).
\]
This is a fixed integer polynomial: $h$ and the exponents defining the index data are fixed by the choice of the universal quadratic source system, while $z,u,y$ are index parameters. Its degree in $x,\mathbf v$ is at most four (in fact four). For every $a_*\in\N$, the total computable map specified after \eqref{eq:quartic-index} yields the triple $(z(a_*),u(a_*),y(a_*))$, and we have proved that for every $x\in\N$,
\[
 x\in W_{a_*}\iff\exists\mathbf v\in(\Z^+)^{47}\,
 [\mathcal U(x,z(a_*),u(a_*),y(a_*),\mathbf v)=0].
\]
Thus \eqref{eq:universal-definition} holds with three index parameters. This completes the proof of Theorem \ref{Th1.2}.
\end{proof}

\section{Degree bounds for eight and nine unknowns}\label{sec:degree-bounds}
\setcounter{equation}{0}

We now use Theorem \ref{Th1.2} to give explicit degree bounds for the two small numbers of unknowns in Theorem \ref{Th1.1} and Corollary \ref{Cor1.1}. Throughout this section, degree means total degree in the input and the counted unknowns, with index parameters held fixed. To reduce the degree, we use a slightly different source parametrization and weights in the relation-combining polynomials. These modifications concern the degree bounds only; the proofs of the preceding theorems do not depend on this section.

\medskip\noindent\textbf{The source polynomials.}
Fix $\nu=47$, put $d_*=5^\nu$, and fix a program index $a_*\in\N$. By Theorem \ref{Th1.2} and its proof, evaluating the index parameters at their computable values for $a_*$ and translating the positive witnesses by one gives an integer polynomial $P(a,z_1,\ldots,z_\nu)$ of degree at most four with
\[
 a\in W_{a_*}\iff\exists z_1\gs0\ldots\exists z_\nu\gs0\,[P(a,z_1,\ldots,z_\nu)=0]\quad(a\in\N).
\]
Here $a$ denotes the input and the coefficients of $P$ are total computable functions of the index $a_*$.
We may order the witnesses so that $z_\nu$ corresponds to the base $B$ in \eqref{eq:quartic-base}, translated by one. Thus every zero with nonnegative witnesses has $z_\nu\gs1$, since $B=\kappa b^2>1$ in that system. Moreover $P(a,0,\ldots,0)>0$ for every $a\in\N$: at the all-one tuple of the positive witnesses the residual of $B=\kappa(x+\ve)^2$ is nonzero. Here we use the sum of squares in the proof of Theorem \ref{Th1.2}, not an arbitrary polynomial witnessing universality. In particular no additional source unknown is needed to ensure $P(0,\ldots,0)>0$.

For the degree bounds it is useful to allow a positive integer $b$ exceeding $a$, without requiring the parametrization \eqref{eq:b-def}. We give the relevant variant of the construction in \cite[Theorem~3.1]{Sun21}. Write
\[
 P(a,\mathbf z)=\sum_{|\mathbf i|\ls4}p_{\mathbf i}\,a^{i_0}\prod_{j=1}^\nu z_j^{i_j},
 \quad \mathbf i=(i_0,\ldots,i_\nu),\quad |\mathbf i|=\sum_{j=0}^{\nu}i_j.
\]
Let $L_*\gs1$ be the maximum of 1 and the absolute values of the coefficients $p_{\mathbf i}$. Let $\beta_*$ be the least power of two exceeding $2\cdot4! L_*(\nu+2)^4+4$, and define
\begin{align*}
 \widehat B&=\beta_* b^4,\\
 \widehat D(Z)&=\sum_{|\mathbf i|\ls4}(4-|\mathbf i|)!\prod_{j=0}^{\nu}i_j!\,p_{\mathbf i}\,
 Z^{5d_*-\sum_{j=0}^{\nu}i_j5^j},\\
 \widehat M&=(\widehat B-1)\sum_{j=0}^{d_*}\widehat B^j-(b-1)\sum_{j=1}^{\nu}\widehat B^{5^j},\\
 N_0&=\widehat B^{d_*+1},\quad N_1=4 \widehat B^{9d_*+1},\quad \widehat N=N_0N_1,\quad \Cscr=b \widehat B^{d_*},\\
 \widehat J&=2(1+a\widehat B+g)^4\widehat D(\widehat B)+\sum_{j=0}^{9d_*}\widehat B^{j+1},\\
 \widehat S&=g+\widehat JN_0,\quad \widehat T=\widehat M+(\widehat B-2) \widehat B^{5d_*}N_0,\\
 \widehat R&=(\widehat S+\widehat T+1)\widehat N+\widehat T+1,\quad X=(\widehat N-1)\widehat R,\quad Y=\widehat N^2.
\end{align*}
All the exponents are fixed. The coefficient list of $P$ and $\beta_*$ form a finite tuple of index parameters of fixed length. The above rule makes this tuple a total computable function of $a_*$. Computing $\beta_*$ from $a_*$ does not affect the degree bound or require a new counted unknown.

Here are the properties of these polynomials that we shall need. Suppose that $b>a\gs0$, $b\in\Z^+$ and $0\ls g<2\Cscr$. Then
\begin{equation}\label{eq:degree-source-bounds}
 4\mid X,\quad X\gs3b,\quad Y\gs\max\{b,2^8\},\quad Y\ls X+1.
\end{equation}
If also $b\in2\uparrow$, then $Y\in2\uparrow$ and
\begin{equation}\label{eq:degree-source-iff}
 Y\mid\binom{2X}{X}\iff\exists z_1,\ldots,z_\nu\in\N:
 \begin{cases}
 g=\sum_{j=1}^{\nu} z_j \widehat B^{5^j},\\
 0\ls z_j<b\ (1\ls j\ls\nu),\\
 P(a,z_1,\ldots,z_\nu)=0.
 \end{cases}
\end{equation}
On the right the digits, if they exist, are unique. In particular, given an actual source solution, we may take arbitrarily large powers of two $b$ (including powers of four) and the corresponding digit code $g$; for these choices
\begin{equation}\label{eq:degree-source-lower}
 \widehat B^{d_*}\ls g<\Cscr,
\end{equation}
since $z_\nu\gs1$.

For completeness, the verification of this variant is as follows. The exponents in $\widehat D$ are distinct, by uniqueness of base-five digits of weight at most four. Its coefficients have absolute values at most $4!L_*$. Its leading coefficient is $4! P(0,0,\ldots,0)>0$, so $\widehat D(\widehat B)>0$ for every positive $b$ by the bound on $\beta_*$. For the range above, $g<N_0$ and
\[
 1+a\widehat B+g\ls(a+2b+1)\widehat B^{d_*}\ls3b \widehat B^{d_*}.
\]
Consequently
\[
 0<\widehat J\ls(\widehat B-2)\widehat B^{4d_*}\sum_{j=0}^{5d_*}\widehat B^j
                  +\sum_{j=0}^{9d_*}\widehat B^{j+1}<N_1.
\]
Also $0\ls \widehat M<N_0$ and $0\ls(\widehat B-2)\widehat B^{5d_*}<N_1$. Thus $0\ls \widehat S,\widehat T<\widehat N$ and $\widehat R\gs \widehat N+1$. The constant digit of $\widehat M$ is $\widehat B-1$, so $\widehat B\mid \widehat T+1$ and $\widehat B\mid \widehat R$. Since $4\mid \widehat B$, the first assertion in \eqref{eq:degree-source-bounds} follows. The remaining bounds follow from $X\gs \widehat N^2-1$ and the size of $\widehat B,\widehat N$.

When $b$ is a power of two, the identity
\[
 X=(\widehat S+\widehat T)\widehat N^2+(\widehat N-1-\widehat S)\widehat N+\widehat N-1-\widehat T
\]
and Kummer's theorem give $Y\mid\binom{2X}{X}$ if and only if $\tau_2(\widehat S,\widehat T)=0$, as in the proof of Theorem \ref{Th1.2}. This carry condition splits at $N_0$ into the conditions for $g,\widehat M$ and $\widehat J,(\widehat B-2)\widehat B^{5d_*}$. The first mask forces precisely the digit form of $g$ in \eqref{eq:degree-source-iff}. For those digits, the coefficient of $Z^{5d_*}$ in
\[
 \bigg(1+aZ+\sum_{j=1}^\nu z_jZ^{5^j}\bigg)^4\widehat D(Z)
\]
is $4! P(a,\mathbf z)$, by the multinomial theorem and uniqueness of base-five digits. Each coefficient is strictly between $-\widehat B/2$ and $\widehat B/2$, since its absolute value is at most $4!L_*((\nu+2)b)^4$. The integer $\widehat J/2$ therefore has ordinary base-$\widehat B$ digits given by these coefficients plus $\widehat B/2$, all in $[1,\widehat B-1]$. The second mask condition, after division of both summands by two, holds precisely when the digit at $5d_*$ is $\widehat B/2$. This is precisely $P(a,\mathbf z)=0$. These arguments prove both implications of \eqref{eq:degree-source-iff}. The usual upper bound for a code with digits less than $b$ is $g<b\widehat B^{d_*}$, and a positive highest digit gives \eqref{eq:degree-source-lower}.

\medskip\noindent\textbf{The two parametrizations and the modified approximation range.}
For the eight-unknown bound over $\N$, put $b=f+1$ with $f\in\N$. For the nine-unknown bound over $\Z$, put $b=f^2$ with $f\in\Z$. We require $f\ne0$ by the factor $f^2$ in the norm \eqref{eq:O0}. In both cases use the source polynomials just defined and keep \eqref{eq:A-def}--\eqref{eq:bridge-def}. In \eqref{eq:range} replace the two range expressions by
\begin{equation*}
 R_g'=4b^2(g+1)(2\Cscr-g)-a^2(2\Cscr+1)^2\ \ \t{and}\ \  \kappa'=4b^2(2\Cscr+1)^2.
\end{equation*}
Use $R_g',\kappa'$ instead of $R_g,\kappa$ in \eqref{eq:norm-poly}, and use \eqref{eq:O0} with the resulting remainder. Denote the resulting integer polynomial again by $O_0$ in this section. The quantity governing the approximation, for $f\ne0$, is now
\begin{equation*}
 \Delta'=\frac{R_g'}{\kappa'}
 =\frac{(g+1)(2\Cscr-g)}{(2\Cscr+1)^2}-\frac{a^2}{4b^2}.
\end{equation*}
It is at most $1/4$. Its positivity forces $0\ls g<2\Cscr$ and $a<b$. In the positive direction, choose $b$ to be a sufficiently large power of two (a power of four for the integer parametrization), with the source code in \eqref{eq:degree-source-lower}. Then
\[
 \frac{(g+1)(2\Cscr-g)}{(2\Cscr+1)^2} >\frac{g}{9\Cscr}\gs\frac1{9b},
\]
and hence $\Delta'>0$ for all sufficiently large such $b$. This is the reason for using a positive highest source digit; the new penalty $a^2/(4b^2)$ does not obstruct the positive implication.

The norm factorization \eqref{eq:norm-factor} is unchanged with $\kappa',\Delta'$ in place of $\kappa,\Delta$. As in Lemma \ref{lem:norm}, its positivity first forces $f,k\ne0$; then $b>0$, $Y>0$ and $A\gs64$, so the negative conjugate is nonpositive. We recover $\Delta'>0$ and one of the approximations \eqref{eq:near} with $\Delta'$. After this, \eqref{eq:degree-source-bounds} gives all size and parity conditions required in Lemma \ref{lem:size} and Theorem \ref{thm:modified-reverse}. The proof of Lemma \ref{lem:merged-arithmetic} uses only the bounds on $b,X,Y$, positivity of $b$, and the fact that $Y$ is a power of two whenever $b$ is. It therefore applies to the present source polynomials. Conversely, by choosing $k$ sufficiently large after $b,g$ are fixed, the approximation error is less than $\sqrt{\Delta'}$. The proof of Proposition \ref{prop:bridge} now gives the same two implications with the new source and norm. The proof of Theorem \ref{Th1.1}(i) has positive choices of $f,k,x$, nonnegative $g$, positive $h,l$ (use $C=u_{X+1}(A)>X+1$ and $P_X^{(+1)}(A)>0$), and a positive $y$ by taking $L$ sufficiently large in the proof of Theorem \ref{thm:full-interpolation}. Consequently its forward implication here uses only nonnegative witnesses. For the integer parametrization the proofs of Theorem \ref{Th1.1}(ii), including the choice $\xi=zx$ and the divisibility $zx\mid\Pi(y)$, remain valid.

\medskip\noindent\textbf{A smaller-degree relation combination.}
We recall an elementary useful fact about signed sums of radicals. Suppose integer radicands $V_1,\ldots,V_s$ and integer weights $w_1=1,w_2,\ldots,w_s$ satisfy
\begin{equation}\label{eq:degree-weights}
 |w_j|>\sum_{i<j}|w_i|\sqrt{|V_i|}\quad(2\ls j\ls s).
\end{equation}
If a sum $\sum_i\varepsilon_i w_i\sqrt{V_i}$ with $\varepsilon_i\in\{\pm1\}$ is rational, all the $V_i$ are integer squares. Indeed, write each nonzero radicand as a square times a signed squarefree integer. The distinct squarefree radicals (including 1) are linearly independent over $\Q$: in the multiquadratic field they belong to different characters of its group of sign automorphisms. In any nonsquare class the nonzero contribution with largest index has absolute coefficient at least $|w_j|$ after factoring out the common squarefree radical, whereas the sum of the absolute coefficients of the lower contributions is smaller by \eqref{eq:degree-weights}. Thus that class cannot cancel. This proves the assertion, also for negative radicands.

For the nonnegative witnesses define positive polynomial bounds as follows (all $A,B,C,D,E,F,G^*,\Omega,H^*,I^*$ refer to the modified certificate with $\xi=x$):
\begin{gather*}
 \widehat Z=(T_0+2)M_0 A_0+2,\quad \widehat W=\widehat Z+1,\\
 C^+=B+(A+2)h,\quad D^+=(A^2+4)(C^+)^2+4,\\
 E^+=(C^+)^2D^+x,\quad F^+=4(A^2+4)(E^+)^2+1,\\
 G^+=2(A+2)^3(E^+)^2+1,\\
 H^+=C^++(B+C^+)F^++16E^+F^+y,\\
 Z_L=((A+2)C^++2)(2(A+2)E^++1)((G^++1)H^++1).
\end{gather*}
Here the positivity of $\widehat D(\widehat B)$ and $g,a\gs0$ guarantees $X>0$ and thus $B,M_0,T_0>0$ even before the range is recovered. Also $A,A_0,h,x,y\gs0$. The triangle inequality in each defining expression gives
\[
 \widehat Z\gs\sqrt{|D_0|}\ \ \t{and}\ \ Z_L\gs\sqrt{|\mathcal D(x,y)|}.
\]
For example, $E^+\gs|E|$, $H^+\gs|H^*|$, and $((G^++1)H^++1)^2\gs|I^*|$; the other two factors of $Z_L$ bound $\sqrt{|D|}$ and $\sqrt{|F|}$. Put $V_1=D_0,V_2=\mathcal D(x,y)$, $R=x^2O_0$, and set
\[
 U_*=1+\widehat Z+\widehat W Z_L\ \ \t{and}\ \ \gamma=2R-1.
\]
The polynomial
\begin{equation}\label{eq:degree-P8}
 \prod_{\varepsilon_1,\varepsilon_2\in\{\pm1\}}
 \big(m+1-\gamma(U_*+\varepsilon_1\sqrt{V_1}+\varepsilon_2 \widehat W\sqrt{V_2})\big)
\end{equation}
expands to an integer polynomial, since it is invariant under both sign changes. On the domain of nonnegative witnesses its zero implies that the signed sum is rational; \eqref{eq:degree-weights} implies both square conditions. The expression in parentheses after $\gamma$ is then at least 1, so $m+1>0$ implies $R>0$. Conversely the two squares and $R>0$ allow a positive integer choice of $m+1$, using the two positive square roots. Thus \eqref{eq:degree-P8} has exactly the required equivalence, without multiplying over sign choices of the seven original witnesses.

For all integer witnesses put
\[
 V_1=D_0,\quad V_2=D_{\rm pos}=(3O_0-4)z^2+1,\quad V_3=\mathcal D(zx,y),\quad \Sigma=zx.
\]
Choose weights
\[
 w_2=16(M_0^2+1)A_0+4\ \ \t{and}\ \ w_3=w_2(4V_2+2).
\]
For arbitrary integer $M_0,A_0$, using $T_0=2M_0+2$ gives
\[
 \sqrt{|D_0|}\ls (|T_0|+2)|M_0A_0|+2\ls4(M_0^2+1)|A_0|+2.
\]
If $A_0=0$, then $|w_2|=4>\sqrt{|D_0|}=2$; otherwise the reverse triangle inequality gives $|w_2|>\sqrt{|D_0|}$. Since $|4v+2|\gs1+\sqrt{|v|}$ for every integer $v$, the weight $w_3$ also satisfies \eqref{eq:degree-weights}. Use the polynomial
\begin{equation}\label{eq:degree-Q9}
 \prod_{\varepsilon_1,\varepsilon_2,\varepsilon_3\in\{\pm1\}}
 \big(\Sigma m-\Pi(y)+\Sigma(\varepsilon_1\sqrt{V_1}+\varepsilon_2w_2\sqrt{V_2}+\varepsilon_3w_3\sqrt{V_3})\big).
\end{equation}
This is again an integer polynomial after expansion. If $\Sigma=0$ its value is $\Pi(y)^8\ne0$. Otherwise a zero gives a rational signed sum, hence all three square conditions by \eqref{eq:degree-weights}, and then $\Sigma\mid\Pi(y)$. Conversely these conditions give an integer choice of $m$. We have thus obtained the same positive and reverse implications as in the proof of Theorem \ref{Th1.1}(ii).

\medskip\noindent\textbf{The degree calculation.}
Let $s_b=\deg b$, so $s_b=1$ for \eqref{eq:degree-P8} and $s_b=2$ for \eqref{eq:degree-Q9}. Put
\[
 p=(40d_*+8)s_b,\qquad c_*=(4d_*+1)s_b,\qquad v=2p+s_b+1.
\]
The bounds on the source expressions follow directly from their definitions: $\deg \widehat B=4s_b$, $\deg \widehat N=p$, $\deg\Cscr=c_*$, $\deg \widehat J\ls (9d_*+1)4s_b$ and $\deg \widehat S,\deg \widehat T\ls p$. For the last bound on $\widehat J$, the first summand has degree at most $4(1+4s_b)+20d_*s_b$, which is at most $(9d_*+1)4s_b$. Therefore $\deg \widehat R\ls2p$, $\deg X\ls3p$ and $\deg Y\ls2p$. The subsequent degree bounds are displayed below. Write $e_\xi=1$ or 2 according as $\xi=x$ or $\xi=zx$.
\[
\begin{array}{c|l@{\qquad}c|l}
 \text{Expression}&\text{Degree bound}&\text{Expression}&\text{Degree bound}\\\hline
 A_0&v&A&2v\\
 C&2v+1&D_0&d_0'=16p+2s_b+2\\
 D&8v+2&E&12v+4+e_\xi\\
 F&28v+8+2e_\xi&G^*&30v+8+2e_\xi\\
 H^*&40v+13+3e_\xi&I^*&140v+42+10e_\xi\\
 \mathcal D(\xi,y)&\ell_{e_\xi}=176v+52+12e_\xi&&
\end{array}
\]
These are upper bounds; we do not need to determine leading-term cancellations. In the norm computation give $t$ weight $2v$ and each counted variable weight 1. Reduction by $t^2-At+1$ does not increase this weighted degree. We have $\deg\kappa'\ls 2s_b+2c_*$, $\deg R_g'\ls2+2c_*$, and the weighted degrees of $s(t),H(t),J_0$ are at most $4v,6v+1,s_b+6v+1$, respectively. Both terms of the modified \eqref{eq:norm-poly} therefore have weighted degree at most
\[
 w_*=2(2s_b+2c_*)+32v+4.
\]
For its remainder $r_0+r_1t$ we have $\deg r_0\ls w_*$ and $\deg r_1\ls w_*-2v$. The norm in \eqref{eq:O0}, including $f^2k^2$, has degree at most
\[
 o_*=2w_*+4=8s_b+8c_*+64v+12.
\]

For \eqref{eq:degree-P8}, the polynomial bounds $\widehat Z,\widehat W$ have degrees at most $d_0'/2$ and $Z_L$ has degree at most $\ell_1/2$, by their displayed product definitions. Thus $\deg U_*\ls(d_0'+\ell_1)/2$. Every factor of the sign product, regarding a square root as having half the degree of its radicand, has degree at most $o_*+2+(d_0'+\ell_1)/2$. Expanding the four factors and eliminating even powers of the radicals yields the degree bound
\[
 4\big(o_*+2+(d_0'+\ell_1)/2\big)=50048d_*+11456\quad(s_b=1).
\]
For \eqref{eq:degree-Q9}, the weight $w_2$ has degree at most $d_0'/2$ and $w_3$ at most $d_0'/2+o_*+2$. The degree of $\Pi(y)$ is six, and $\deg\Sigma=2$. Each factor has degree at most $2+d_0'/2+o_*+2+\ell_2/2$, which dominates the other terms. Expanding the eight factors therefore gives the bound
\[
 8\big(2+d_0'/2+o_*+2+\ell_2/2\big)=200192d_*+44296\quad(s_b=2).
\]
The sign products have integer coefficients and use the finite index tuple specified above. This tuple is computed from $a_*$ by the map of Section \ref{sec:quartic}, the coefficient calculation for $P$, and the least-power-of-two rule for $\beta_*$. Every subsequent polynomial operation has fixed exponents. Consequently, for all $a_*\in\N$ and $a\in\N$, the nonnegative-integer and integer zero conditions in \eqref{eq:degree-P8} and \eqref{eq:degree-Q9}, respectively, are each equivalent to $a\in W_{a_*}$, with the corresponding index tuple supplied by a total computable function of $a_*$. This verifies \eqref{eq:universal-definition} on the two witness domains, with the degrees just obtained, and proves Corollary \ref{Cor1.2}.


\begin{thebibliography}{99}

\bibitem {C80} N. Cutland, {\it Computability: An Introduction to Recursive Function Theory}, Cambridge Univ. Press, Cambridge, 1980.


\bibitem{DPR}
M.~Davis, H.~Putnam and J.~Robinson,
\emph{The decision problem for exponential Diophantine equations},
Ann.\ of Math. (2) \textbf{74} (1961), 425--436.

\bibitem {J82} J. P. Jones,  {\it Universal Diophantine equation},
 J. Symbolic Logic 47 (1982), 549--571.


\bibitem{M70}
Y.~Matiyasevich,
\emph{Enumerable sets are Diophantine},
Dokl. Akad. Nauk SSSR \textbf{191} (1970), 279--282;
English translation, Soviet Math. Dokl. \textbf{11} (1970), 354--357.

\bibitem{MR75}
Y.~Matiyasevich and J.~Robinson,
\emph{Reduction of an arbitrary Diophantine equation to one in 13 unknowns},
Acta Arith. \textbf{27} (1975), 521--553.

\bibitem {M77}  Y. Matiyasevich, {\it Some purely mathematical results inspired by mathematical logic},
in: Logic, Foundations of Mathematics and Computability Theory
(London, Ont., 1975). Reidel, Dordrecht, 1977, Part I, 121--127.

\bibitem {M81} Y. Matiyasevich,  {\it Primes are nonnegative values of a polynomial in 10 variables}.
 J. Soviet Math. {\bf 15} (1981), 33--44.
 
\bibitem {MS} Y. Matiyasevich and Z.-W. Sun, {\it Undecidability on Diophantine equations over $\mathbb Z[i]$ with 20 unknowns}, J. Number Theory {\bf 290} (2027), 210--219.

\bibitem {MPR} A. B. Matos, L. Paolini, L. Roversi,
{\it The fixed point problem of a simple reversible language},
Theoret. Comput. Sci. {\bf 813} (2020), 143--154.

\bibitem {N96} M. B. Nathanson, Additive Number Theory: The
Classical Bases, Grad. Texts in Math., vol. 164,
New York: Springer, 1996.

\bibitem {P69} H. Putnam H, {\it An unsolvable problem in number theory},
 J. Symbolic Logic {\bf 25} (1960), 220--232.

 \bibitem{RGK}   J. Richter-Gebert, U. H. Kortenkamp,
 {\it  Complexity issues in dynamic geometry},
in: Foundations of Computational Mathematics (Hong Kong, 2000),
World Sci., New Jersey, 2002, pp.\, 355--404.

\bibitem{Sk} T. Skolem, {\it \"Uber die Nicht-charakterisierbarkeit der Zahlenreihe mittels endlich oder abz\"ahlbar unendlich vieler Aussagen mit ausschliesslich Zahlenvariablen},
    Fund. Math. {\bf 23} (1934), 150--161.

\bibitem{Sun92}
Z.-W.~Sun,
\emph{Reduction of unknowns in Diophantine representations},
Sci. China Ser. A \textbf{35} (1992), no.~3, 257--269.

\bibitem {S92b} Z.-W. Sun, {\it  A new relation-combining theorem and its application},
 Z. Math. Logik Grundlag. Math. {\bf 38} (1992), 209--212.


\bibitem{Sun21}
Z.-W.~Sun,
\emph{Further results on Hilbert's tenth problem},
Sci. China Math. \textbf{64} (2021), no.~2, 281--306;
arXiv:1704.03504v7.

\bibitem{S24}
Z.-W. Sun,
Fibonacci Numbers and Hilbert's Tenth Problem,
Harbin Institute of Technology Press, Harbin, 2024.

\bibitem {T85} S. P. Tung,  {\it On weak number theories},
 Japan. J. Math. (N.S.) {\bf 11} (1985), 203--232.

\end{thebibliography}
 \end{document}